\documentclass[12pt,a4paper]{article}
\usepackage{amsmath,amstext,amssymb,amscd, color}
\usepackage{graphicx}

\newcommand{\Xcomment}[1]{}

\newtheorem{theorem}{Theorem}[section]
\newtheorem{lemma}[theorem]{Lemma}
\newtheorem{corollary}[theorem]{Corollary}

\newtheorem{prop}[theorem]{Proposition}

\newcommand{\SEC}[1]{\ref{sec:#1}}  

\makeatletter \@addtoreset{equation}{section} \makeatother

\newenvironment{proof}{\noindent{\textbf{Proof}}~}%
{\hfill$\qed$\medskip}

\def\qed{ \ \vrule width.1cm height.3cm depth0cm}

\newenvironment{numitem1}{\refstepcounter{equation}\begin{enumerate}%
\item[(\thesection.\arabic{equation})]}{\end{enumerate}}

\newcommand{\refeq}[1]{(\ref{eq:#1})}  

 \makeatletter
\renewcommand{\section}{\@startsection{section}{1}{0pt}%
{-3.5ex plus -1ex minus -.2ex}{2.3ex plus .2ex}%
{\normalfont\Large}}
 \makeatother

 \makeatletter
\renewcommand{\subsection}{\@startsection{subsection}{2}{0pt}%
{-3.0ex plus -1ex minus -.2ex}{1.5ex plus .2ex}%
{\normalfont\normalsize\bf}}
 \makeatother

 \makeatletter
\renewcommand{\subsubsection}{\@startsection{subsubsection}{2}{0pt}%
{-2.0ex plus -1ex minus -.2ex}{-2.0ex plus .2ex}%
{\normalfont\normalsize\underline}}
 \makeatother

\def\rest#1{_{\,\vrule height 1.5ex width 0.05em depth 0pt\, #1}}

\def\Rset{{\mathbb R}}
\def\Zset{{\mathbb Z}}
\def\Zset{{\mathbb Z}}
\def\Rsetge{{\mathbb R_{\ge 0}}}
\def\Zsetge{{\mathbb Z_{\ge 0}}}

\def\Cscr{{\cal C}}

\def\Fscr{{\cal F}}

\def\Hscr{{\cal H}}

\def\Lscr{{\cal L}}
\def\Mscr{{\cal M}}

\def\Pscr{{\cal P}}

\def\tilde{\widetilde}
\def\hat{\widehat}
\def\bar{\overline}
\def\eps{\varepsilon}

\def\dist{{\rm dist}}
\def\inter{{\rm Int}}
\def\bd{{\rm bd}}

\begin{document}

\baselineskip=15pt
\parskip=2pt

\title{An efficient algorithm for packing cuts and (2,3)-metrics
              in a planar graph with three holes}
\author{
Alexander V. Karzanov
\thanks {Institute for System Analysis at FRC Computer Science and
Control of the RAS, 9, Prospect 60 Let Oktyabrya, 117312 Moscow, Russia; email:
akarzanov7@gmail.com. } }

\date{}

\maketitle

 \begin{abstract}
We consider a planar graph $G$ in which the edges have nonnegative integer
lengths such that the length of every cycle of $G$ is even, and three faces are
distinguished, called \emph{holes} in $G$. It is known that there exists a
packing of cuts and (2,3)-metrics with nonnegative integer weights in $G$ which
realizes the distances within each hole. We develop a strongly polynomial
purely combinatorial algorithm to find such a packing.
 \end{abstract}

\smallskip



\section{Introduction} \label{sec:intr}

In combinatorial optimization there are known packing problems on cuts and
metrics that are related via a sort of polar duality to popular
multi(commodity)flow demand problems in graphs. (For a discussion on such a
relationship and some appealing examples, see, e.g.,
\cite[Sect.~4]{kar90b},\cite[Sects.~71,72,74]{sch_book}. The simplest example
is the polar duality between the classical problems of finding an $s$--$t$ flow
of a given value in a capacitated graph $G$ and finding a shortest $s$--$t$
path in $G$ with nonnegative lengths $\ell$ of edges, or, equivalently, finding
a maximal packing of cuts  separating the vertices $s$ and $t$ in $(G,\ell)$.)

In this paper we consider a planar graph $G=(V,E)$ embedded in the plane in
which the edges $e\in E$ have \emph{nonnegative lengths} $\ell(e)\in\Rsetge$
and a subset $\Hscr$ of faces of $G$, called \emph{holes}, is distinguished.
Also we are given (implicitly) a certain set $\Mscr$ of metrics on $V$. Then
the (fractional) \emph{problem of packing metrics realizing the distances on
the holes} for $G,\ell,\Hscr,\Mscr$ consists in the following:

\begin{description}
\item{\textbf{PMP}}: ~Find metrics $m_1,\ldots,m_k\in\Mscr$ and weights
$\lambda_1,\ldots,\lambda_k\in \Rsetge$ such that:
  \begin{eqnarray}
  \lambda_1 m_1(e)+\ldots +\lambda_k m_k(e) \le \ell(e)
            &&\mbox{for each $e\in E$}; \quad\mbox{and} \label{eq:P1} \\
 \lambda_1 m_1(st)+\ldots +\lambda_k m_k(st) = \dist_{G,\ell}(st) &&
     \mbox{for all $s,t\in V_H$, $H\in\Hscr$}. \label{eq:P2}
  \end{eqnarray}
  \end{description}

\noindent Hereinafter we use the following terminology and notation:

(a) ~when it is not confusing, a pair $(x,y)$ of vertices may be denoted as
$xy$;

(b) ~for a face $F$, its \emph{boundary} (regarded as a graph) is denoted by
$\bd(F)=(V_F,E_F)$;

(c) ~a \emph{metric} on $V$ is meant to be a function $m: V\times V\to
\Rsetge\cup\{\infty\}$ satisfying $m(xx)=0$, $m(xy)=m(yx)$, and $m(xy)+m(yz)\ge
m(xz)$ for all $x,y,z\in V$ (admitting $m(xy)=0$ for $x\ne y$);

(d) ~$\dist_{G,\ell}(uv)$ denotes the \emph{distance} in $(G,\ell)$ between
vertices $u,v\in V$, i.e., the minimum length $\ell(P):=\sum(\ell(e)\colon e\in
E_P)$ of a path $P$ connecting vertices $u$ and $v$ in $G$ (where $E_P$ is the
set of edges in $P$); in particular, $\dist_{G,\ell}$ is a metric.
 \medskip

Typically the class $\Mscr$ of metrics figured in PMP is described by fixing
one or more graphs $K=(V_K,E_K)$ and ranging over arbitrary mappings
$\gamma:V\to V_K$. Then each $\gamma$ generates the metric $m_\gamma$ on $V$ by
setting $m_\gamma(xy):=\dist_{K,1}(\gamma(x)\gamma(y))$ for $x,y\in V$ (where 1
stands for the all-unit function on $E_K$). Two special cases of metrics are
important for us:

(i) when $K=K_2$ (the graph with two vertices and one edge), $m=m_\gamma$ is
called a \emph{cut-metric}; in other words, $m$ is generated by a partition
$\{V_1,V_2\}$ of $V$ and establishes distance 0 inside each of $V_1$ and $V_2$,
and 1 between the elements of these subsets;

(ii) when $K=K_{2,3}$ (the complete bipartite graph with parts of 2 and 3
vertices), $m_\gamma$ is called a \emph{(2,3)-metric}; it is generated by a
partition of $V$ into five subsets $S_1,S_2,S_3,T_1,T_2$ and establishes
distance 0 inside each of them, 1 between $S_i$ and $T_j$ ($i=1,2,3$, $j=1,2$),
and 2 otherwise.

Nontrivial integrality results on PMP, mentioned below, have been obtained when
$|\Hscr|$ is ``small'' and the edge length function $\ell$ is \emph{cyclically
even}, which means that $\ell$ is integer-valued and the length $\ell(E_C)$ of
any cycle $C$ in $G$ is even. (For a function $f:S\to\Rset$ and a subset
$S'\subseteq S$, we write $f(S')$ for $\sum(f(e)\colon e\in S')$.)

  \begin{numitem1} \label{eq:HST}
If $\ell$ is cyclically even and $\Mscr$ is the set of cut-metrics on $V$, then
PMP has an \emph{integer solution} (i.e., having integer weights of all
metrics) when $|\Hscr|=1$~\cite{HST} and when $|\Hscr|=2$~\cite{sch}.
 \end{numitem1}
   \begin{numitem1} \label{eq:H=3}
If $|\Hscr|=3$, ~$\ell$ is cyclically even, and $\Mscr$ is formed by cut- and
(2,3)-metrics on $V$, then PMP has an integer solution~\cite{K3hI}.
  \end{numitem1}
   \begin{numitem1} \label{eq:H=4}
If $|\Hscr|=4$, ~$\ell$ is cyclically even, and $\Mscr$ is formed by cut-,
(2,3)- and $4f$-metrics on $V$, then PMP has an integer solution, where a
$4f$-\emph{metric} is generated by a mapping $\gamma:V\to V_K$ with $K$ being a
planar graph with four faces~\cite{K3hI}.
  \end{numitem1}

In fact, the existence of a solution to PMP with real-valued weights $\lambda$
in cases $|\Hscr|=1,2,3$ can be immediately concluded, via polar duality, from
solvability criteria for corresponding fractional multiflow demand problems
given in~\cite{OS},\cite{ok},\cite{K3hII}, respectively, and the essence
of~\refeq{HST}--\refeq{H=4} is just the existence of integer solutions when the
lengths of edges are cyclically even (or, weaker, a half-integer solution when
the lengths are integer-valued). The proof for $|\Hscr|\le 2$ given
in~\cite{sch} is constructive and can be turned into a pseudo-polynomial
algorithm. A strongly polynomial combinatorial algorithm for finding a solution
formed by cut-metrics with integer weights when $|\Hscr|\le 2$ and $\ell$ is
cyclically even is given in~\cite{kar90}.

The purpose of this paper is to devise a strongly polynomial combinatorial
algorithm for $|\Hscr|=3$.

  \begin{theorem} \label{tm:H3alg}
Let $G=(V,E)$ be a planar graph with cyclically even lengths $\ell(e)$ of edges
$e\in E$ and let $\Hscr$ be three distinguished faces of $G$. Then one can
find, in strongly polynomial time, cut- or (2,3)-metrics $m_1,\ldots,m_k$ and
nonnegative integers $\lambda_1,\ldots,\lambda_k\in\Zsetge$
satisfying~\refeq{P1}--\refeq{P2}.
  \end{theorem}

Note that this theorem is analogous, in a sense, to a result for arbitrary
graphs in~\cite{kar90b} where a strongly polynomial combinatorial algorithm is
developed that, given a graph $G=(V,E)$ with cyclically even lengths of edges
and a distinguished set $T\subset V$ of five terminals, finds an integer
packing of cuts and (2,3)-metrics realizing the distance for each pair of
terminals.

Our algorithm yielding Theorem~\ref{tm:H3alg} is given throughout
Sects.~\SEC{simpl}--\SEC{final} (in fact, we give an alternative proof
of~\refeq{H=3}). The main part of the algorithm involves three sorts of good
reductions by cuts, called \emph{Reductions I,II,III} and described in
Sects.~\SEC{red_cuts},\SEC{neclac},\SEC{edge_free}, respectively. Here by a
\emph{reduction by cuts}, we mean finding certain cut-metrics with integer
weights and accordingly reducing the current lengths $\ell$, and we say that
the reduction is \emph{good} if combining these cut-metrics with a solution to
PMP for $G,\Hscr$ and the reduced lengths $\ell'$, we obtain a solution for
$\ell$. During the algorithm, as soon as the length of some edge becomes zero,
this edge is immediately contracted. We explain that Reductions I,II,III are
implemented in strongly polynomial time (but do not care of precisely
estimating and decreasing the time bound). Eventually we obtain a graph formed
by three paths having the same pair of endvertices and equal lengths. Then the
distance in $(G,\ell)$ is represented as the sum of weighted (2,3)-metrics and,
possibly, one cut-metric.


\section{Initial reductions } \label{sec:simpl}

We start with additional terminology, notation and conventions.

1) Paths, cycles, subgraphs and other objects in a planar graph $G=(V,E)$ are
naturally identified with their images in the plane or sphere. A path
$P=(v_0,e_1,v_1, \ldots,e_k,v_k)$ from $x=v_0$ to $y=v_k$ is called an $x$--$y$
\emph{path}. When both $x,y$ belong to (the boundary of) a hole $H\in\Hscr$, we
also say that $P$ is an $H$-\emph{path}. When it is not confusing, we may use
notation $P=v_0v_1\cdots v_k$ (via vertices) or $P=e_1e_2\cdots e_k$ (via
edges). For a simple path $P$, its subpath with endvertices $u$ and $v$ may be
denoted as $P[u,v]$ (as a rule, but not always, we assume that $u,v$ occur in
this order in $P$). Also for paths $P=v_0v_1\cdots v_k$ and $P'=v'_0v'_1\cdots
v'_k$ with $v_k=v'_0$, we write $P\cdot P'$ for the concatenated path
$v_0v_1\cdots v_kv'_1\cdots v'_k$. When $x=y$ and $|E_P|>0$, ~$P$ becomes a
\emph{cycle}. When needed, paths/cycles will be regarded up to reversing.

2) The set of faces of $G$ is denoted by $\Fscr_G$. A face $F$ is regarded as a
\emph{closed} region in the plane or sphere (i.e., including the boundary
$\bd(F)=(V_F,E_F)$), and the \emph{interior} $F-\bd(F)$ of $F$ is denoted as
$\inter(F)$.  The boundary $\bd(F)$ may be identified with the corresponding
cycle. Usually  the unbounded face of $G$ is assumed to be a hole.

3) We say that $V_\Hscr:=\cup(V_H\colon H\in\Hscr)$ is the set of
\emph{terminals}. The other vertices of $G$ are called \emph{inner}. Also we
address the adjective \emph{inner} to the faces in $\Fscr_G-\Hscr$ and to the
edges not contained in the boundaries of holes.

4) Usually we will abbreviate the distance function $\dist_{G,\ell}$ to $d$.
Note that the cyclic evenness of $\ell$ implies that for any edge $e=uv$ of
$G$, the integers $\ell(e)$ and $d(uv)$ have the same parity. The distances
$d(xy)$ for all $x,y\in V$ are computed in the beginning of the algorithm and
updated when needed.
\medskip

We assume that $G,\ell,\Hscr$ satisfy the following conditions, which will
simplify our description, leading to no loss of generality in essence:
\begin{description}
\item{(C1):} ~$G$ is connected and has no loops and parallel edges, and the cycle
$\bd(F)$ is simple for each face $F$; in particular $|V_F|\ge 3$ for each face
$F$.
  \end{description}

\noindent (For otherwise we can make easy reductions of the problem, preserving
the cyclic evenness.) The properties in~(C1) will be default maintained during the
algorithm. One more useful simplification is as follows.
\begin{description}
\item{(OP1)} ~In the current graph, if there appears an edge $e=uv$ with
$\ell(e)=0$, then we immediately contract this edge (identifying the vertices
$u$ and $v$).
\end{description}

Also at the \emph{preprocessing stage} of the algorithm,
operations~(OP2)--(OP5) described below are applied, step by step, in an
arbitrary order. To describe them, we need additional definitions and notation.

For a face $F$, let $\Pi_F$ denote the set of all pairs $x,y\in V_F$ of
vertices in $F$. Due to condition~\refeq{P2}, an important role is played by
the set of terminal pairs $\cup (\Pi_H \colon H\in\Hscr)$, denoted as
$\Pi_\Hscr$. We say that an $s$--$t$ path $P$ in $G$ is an
$\Hscr$-\emph{geodesic} (resp. an $H$-\emph{geodesic} for $H\in\Hscr$) if $P$
is shortest w.r.t. $\ell$ and $st\in\Pi_\Hscr$ (resp. $st\in\Pi_H$). For
vertices $x,y,z\in V$, define the values
  \begin{gather}
\eps(x|yz):=d(xy)+d(xz)-d(yz);\quad\mbox{and}  \nonumber\\
  \Delta(xy):= \min\{d(sx)+d(xy)+d(yt)-d(st)\colon st\in \Pi_\Hscr\} \label{eq:excess}
  \end{gather}
(which are efficiently computed when needed). For brevity we write $\Delta(x)$
for $\Delta(xx)$. Clearly each $\Delta(xy)$ is nonnegative and even, and we say
that $xy$ (resp. $x$) is \emph{tight} if $\Delta(xy)=0$ (resp. $\Delta(x)=0$).
In operations (OP2)--(OP4) we decrease the current length $\ell$, trying to
make the values $\Delta$ as small as possible while preserving the cyclical
evenness and the original distances $d$ on $\Pi_\Hscr$.

\begin{description}
\item{(OP2)} ~Suppose that there exists (and is chosen) a non-tight vertex $x$
(this is possible only if $x$ is inner). Then we decrease the length of each
edge $e\in E$ incident to $x$ by $\min\{\ell(e),\Delta(x)/2\}$ (which is a
positive integer).
 \end{description}
As a result, at least one of the following takes place: (i) $\ell(e)$ becomes 0
for some edge $e$ incident to $x$, or (ii) $\Delta(x)$ becomes 0. In case~(i),
we contract $e$ (by applying (OP1)), and if $\Delta(x)$ is still nonzero,
repeat (OP2) with the same $x$.

\begin{description}
\item{(OP3)}
~Suppose that there exists an edge $e=xy\in E$ with $\Delta(xy)>0$. Then we
reduce the length $\ell(e)$ to the minimal nonnegative integer $\alpha$ so that
$\alpha$ and $\Delta(xy)$ have the same parity and $d(sx)+\alpha+d(yt)\ge
d(st)$ for all $st\in\Pi_\Hscr$.
 \end{description}

\begin{description}
\item{(OP4)}
~Suppose that there are two different vertices $x,y$ in an inner face $F$ which
are not adjacent in $G$ and such that $\Delta(xy)>0$. Then we connect $x,y$ by
edge $e$, inserting it inside $F$ (thus subdividing $F$ into two inner faces),
and assign the length $\ell(e)$ in the same way as in~(OP3).
 \end{description}

Clearly (OP3) and (OP4) preserve both the cyclical evenness of lengths and the
distances on $\Pi_\Hscr$, and we can see that
  \begin{numitem1} \label{eq:ex01}
if the new value $\Delta(xy)$ is still nonzero, then $\ell(e)=1$
  \end{numitem1}
(taking into account that $e$ is contracted when $\ell(e)=0$). Note that if
none of (OP3) and (OP4) is applicable, then each edge $e=xy$ satisfies
$\ell(e)=d(xy)$. Moreover, one can see the following useful property:

\begin{description}
\item{(C2):}~
for each face $F$ and vertices $x,y\in V_F$, the pair $xy$ is tight.
 \end{description}
Indeed, suppose that $xy$ is non-tight for some $x,y\in V_F$. This is possible
only if $xy$ is an edge in an inner face $F$. By~\refeq{ex01}, $\ell(xy)=1$.
Take a vertex $z\in V_F$ different from $x,y$ (existing by~(C1)). Since
$d(zx)+d(zy)+\ell(xy)$ is even and $\ell(xy)=1$, either
$d(zx)=d(zy)+\ell(xy)\ge 1+1=2$ or $d(zy)=d(zx)+\ell(xy)\ge 2$. For
definiteness, assume the former. Then $zx$ is tight (since $\Delta(zx)>0$ would
imply that $z,x$ are connected by edge $e$ with $\ell(e)=d(zx)=1$,
by~\refeq{ex01}). This implies that there exists an $\Hscr$-geodesic passing
$z,x$, and hence an $\Hscr$-geodesic passing $x,y$. Then $\Delta(xy)=0$; a
contradiction.

The final operation is intended for getting rid of ``redundant'' edges.
\begin{description}
\item{(OP5)}~
Suppose that some face $F$ contains a \emph{dominating} edge $e=xy$, which
means that $d(xy)=\ell(P)$, where $P$ is the $x$--$y$ path in $\bd(F)$ not
containing $e$ (so $\bd(F)=P\cup\{e\}$). Then we delete $e$ from $G$ (thus
merging $F$ with the other face containing $e$ and preserving the distance
$d$).
 \end{description}

Let $\ell,d,\Delta$ be the corresponding functions obtained upon termination of
the preprocessing stage. Then~(C1) and~(C2) hold, and
\begin{description}
\item{(C3):} ~no face of $G$ has a dominating edge.
  \end{description}

The preprocessing stage has at most $|E|+2|V|^2$ operations. Indeed, let $\eta$
be the current number of non-tight pairs; then $\eta\le|V|^2$. During the
process, the values of $\ell$ and $d$ are non-increasing, and $d$ preserves on
$\Pi_\Hscr$. Then $\eta$ is non-increasing as well. Moreover, (OP1) and~(OP5)
decrease the current $|E|$, ~(OP2) and~(OP3) decrease $\eta$ and do not
increase $|E|$, and~(OP4) decreases $\eta$ though increases $|E|$ by 1. So the
value $|E|+2\eta$ is monotone decreasing, yielding the desired bound.


\section{Reduction I} \label{sec:red_cuts}

It this section we further simplify $(G,\ell,\Hscr)$ by using the  algorithm
from~\cite{kar90} which finds a packing of cuts realizing the corresponding
distances in the two-hole case.

For $X\subset V$, define $\delta X=\delta_G X$ to be the set of edges of $G$
connecting $X$ and $V-X$, referring to it as the \emph{cut} generated by $X$
(or by $V-X$), and define $\rho X=\rho_\Hscr X$ to be the set of pairs $st\in
\Pi_\Hscr$ \emph{separated} by $X$, i.e., $s\ne t$ and $|\{s,t\}\cap X|=1$. The
cut $\delta X$ is associated with the cut-metric corresponding to the the
partition $\{X,V-X\}$.

Let $\chi^{E'}$ denote the incidence vector of a subset $E'\subset E$, i.e.,
$\chi^{E'}(e)=1$ if $e\in E'$, and 0 if $e\in E-E'$.
 \medskip

\noindent\textbf{Definition.} ~Let $\Cscr$ be a collection of cuts $\delta X$
in $G$ equipped with weights $\lambda(X)\in \Zsetge$. We call $(\Cscr,\lambda)$
\emph{reducible} if the function $\ell':=\ell-\sum(\lambda(X)\chi^{\delta X}:
\delta X\in\Cscr)$ (of \emph{reduced lengths}) is nonnegative and the distance
function $d':=\dist_{G,\ell'}$ satisfies
  \begin{equation} \label{eq:red_dist}
d'(st)=d(st)-\sum(\lambda(X)\colon \delta X\in\Cscr,\, st\in\rho X)
  \quad\mbox{for each $st\in\Pi_\Hscr$}.
  \end{equation}
\noindent We also say that the lengths $\ell'$ are obtained by a \emph{good
reduction} using $(\Cscr,\lambda)$.
\medskip

An advantage from such a reduction is clear: once we succeeded to find a
reducible $(\Cscr,\lambda)$, it remains to solve PMP with $(G,\ell',\Hscr)$.
Indeed, $\ell'$ is cyclically even (since any cycle and cut have an even number
of edges in common), and taking an integer solution to PMP with
$(G,\ell',\Hscr)$ and adding to it the weighted cut metrics associated with
$(\Cscr,\lambda)$, we obtain an integer solution to the original problem
(since~\refeq{P1}--\refeq{P2} for $\ell$ are provided by the nonnegativity
$\ell'$ and relation~\refeq{red_dist}). Also $\ell\mapsto \ell'$ does not
decrease the set of $\Hscr$-geodesics. In particular, the following property
(which will be used to show strongly polynomial complexity of the algorithm)
holds:
\begin{numitem1} \label{eq:zero-zero}
if $\eps(x|st)$ for $st\in\Pi_\Hscr$ or $\Delta(xy)$ (defined
in~\refeq{excess}) is zero before a good reduction, then it remains zero after
the reduction.
  \end{numitem1}

Clearly any subcollection of weighted cuts in a reducible $(\Cscr,\lambda)$ is
reducible as well. Also if a cut $\delta X$ (with unit weight) is reducible and
if the subgraph $\langle X\rangle=\langle X\rangle_G$ of $G$ induced by
$X\subset V$ consists of $k$ components $\langle X_1\rangle,\ldots, \langle
X_k\rangle$, then each cut $\delta X_i$ is reducible as well (in view of
$\chi^{\delta X} = \chi^{\delta X_1}+\cdots+\chi^{\delta X_k}$ and $\rho
X\subseteq \rho X_1\cup \ldots \cup \rho X_k$), and similarly for the
components of $\langle V-X\rangle$. So we always may deal with only those cuts
$\delta X$ for which both subgraphs $\langle X\rangle$ and $\langle V-X\rangle$
are connected, called \emph{simple} cuts. The planarity of $G$ implies that
 \begin{numitem1} \label{eq:simple_cut}
for a simple cut $\delta X$ and any face $F$,  $|\delta X\cap E_F|\in\{0,2\}$.
 \end{numitem1}

The main part of the proof of Theorem~\ref{tm:H3alg} will consist in showing
(throughout Sects.~\ref{sec:red_cuts}--\ref{sec:final}) the following
assertion.
  \begin{prop} \label{pr:without_faces}
When $|\Hscr|=3$, one can find, in strongly polynomial time, a reducible
collection of cuts with integer weights so that the reduction of
$(G,\ell,\Hscr)$ by these cuts results in a triple $(G',\ell',\Hscr')$ where
$|\Hscr'|\le 3$ and $G'$ has no inner faces: $\Fscr_{G'}=\Hscr'$.
  \end{prop}

As a step toward proving this assertion, in the rest of this section we
eliminate one sort of reducible cuts. Unless otherwise is explicitly said, when
speaking of a shortest path (or a geodesic), we mean that it is shortest w.r.t.
the current length $\ell$, or $\ell$-shortest.

Fix a hole $H$ and consider an $H$-geodesic $P$ with ends $s,t\in V_H$.  Let
$\Lscr_H(st)$ denote the pair of $s$--$t$ paths that form the boundary of $H$.
  \medskip

\noindent\textbf{Definitions.} ~For $P$ as above and $\Lscr_H(st)=\{L,L'\}$,
define $\Omega(P,L)$ to be the closed region of the sphere bounded by $P$ and
$L$ and not containing the hole $H$, and define $G_{P,L}= (V_{P,L},E_{P,L})$ to
be the subgraph of $G$ lying in $\Omega(P,L)$. We say that the pair $(P,L)$
(and the region $\Omega(P,L)$) is \emph{of type} $i$ and denote $\tau(P,L):=i$
if $\Omega(P,L)$ contains exactly $i$ holes (then $0\le \tau(P,L)+\tau(P,L')=
|\Hscr|-1$). The pair $(P,L)$ is called \emph{normal} if $L$ is shortest, and
\emph{excessive} otherwise (when $\ell(L)>\dist(st)=\ell(P)$). Also we write
$\tau(P):=\min\{\tau(P,L),\tau(P,L')\}$ and define $\tau(H)$ to be the maximum
$\tau(P)$ over all $H$-geodesics $P$, referring to $\tau(P)$ and $\tau(H)$ as
the \emph{type} of $P$ and $H$, respectively.
 \medskip

In particular, if $|\Hscr|\le 4$ then $\tau(P),\tau(H)\le 1$. In the picture
below, $\Hscr=\{H,H',H''\}$, $\tau(P,L)=0$ and $\tau(P',L')=1$.

\vspace{-0cm}
\begin{center}
\includegraphics[scale=0.9]{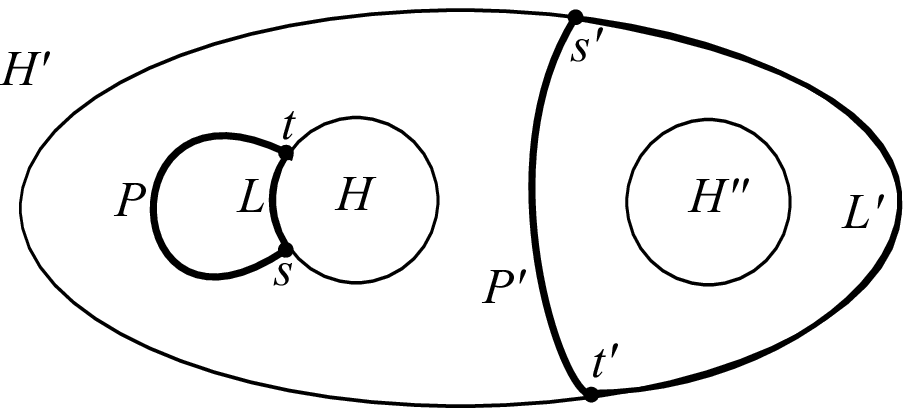}
\end{center}
\vspace{-0cm}

\textbf{An algorithm of eliminating excessive pairs of types 0,1
(Reduction~I).} ~Suppose that a pair $(P,L)$ as above (concerning $H,s,t$) has
type $i\le 1$ and is excessive. For convenience, assume that the region
$\Omega(P,L)$ is bounded. Let $\alpha:=(\ell(L)-d(st))/2$; then $\alpha$ is an
integer $\ge 1$. We wish to make a good reduction by cuts so as to turn $L$
into an $H$-geodesic while preserving the distance $d(st)$.

To this aim, we consider the auxiliary PMP with $(G_{P,L}, \ell_{P,L},
\Hscr')$, where $\ell_{P,L}$ is the restriction of $\ell$ to the edge set
$E_{P,L}$, and $\Hscr'$ consists of the $i$ holes of $\Hscr$ located in
$\Omega(P,L)$ plus the outer face $\bar H$ of $G_{P,L}$ (with the boundary
$P\cup L$).

Since $|\Hscr'|=i+1\le 2$, we can apply the strongly polynomial algorithm
of~\cite{kar90} to find a packing of (simple) cuts $\delta X$ with integer
weights $\lambda(X)>0$ realizing the distances on $\Pi_{\Hscr'}$. From this
packing we extract the set $\Cscr$ of those cuts that meet $\bd(\bar H)$
(twice). Since $P$ is shortest, each cut in $\Cscr$ meets $P$ at most once.
Then $\Cscr$ is partitioned into the set $\Cscr'$ of cuts $\delta X$ with
$|\delta X\cap L|=2$ (and $\delta X\cap P=\emptyset$) and the rest (formed by
the cuts $\delta X$ with $|\delta X\cap L|=|\delta X\cap P|=1$).

One can see that $\sum(\lambda (X)\colon \delta X\in \Cscr')=\alpha$. We assert
that $(\Cscr',\lambda')$ is reducible in the whole $(G,\ell)$, where
$\lambda':=\lambda\rest{\Cscr'}$.

To show this, assume that $s,t\notin X$ for each $\delta X\in\Cscr'$ (since
$|\delta X\cap L|=2$ and the generating set $X$ of this cut can be taken up to
the complement to $V_{P,L}$). Then $X\cap P=\emptyset$, and $X$ generates the
same cut in both $G_{P,L}$ and $G$.

Define $\ell':=\ell(e)-\sum(\lambda(X)\chi^{\delta X} \colon \delta
X\in\Cscr')$. Let $\tilde d$ be the distance within the subgraph $G_{P,L}$ with
the lengths $\ell'(e)$ of edges $e\in E_{P,L}$. The fact that the path $P$
separating $G_{P,L}$ from the rest of $G$ is $\ell$-shortest easily implies
that any $\ell$-shortest path in $G_{P,L}$ is $\ell$-shortest in the whole $G$.
Also $(\Cscr',\lambda')$ is reducible in $(G_{P,L},\ell_{P,L},\Hscr')$, the
cuts in $\Cscr'$ do not meet $P$, and $\sum(\lambda (X)\colon \delta X\in
\Cscr')=\alpha$. These properties imply:
\begin{numitem1} \label{eq:PLshort}
in $G_{P,L}$, both $P,L$ are $\ell'$-shortest, and $\ell'(L)=\tilde
d(st)=d(st)=\ell(P)$; and
 \end{numitem1}
 \begin{numitem1} \label{eq:pqdist}
any $pq\in \Pi_{\Hscr'}$ satisfies
 $\tilde d(pq)=d(pq)-\sum(\lambda(X)\colon \delta\in\Cscr',\,
pq\in\rho_{\Hscr'}(X))$.
  \end{numitem1}

\begin{lemma} \label{lm:Cpreduc}
$(\Cscr',\lambda')$ is reducible for $(G,\ell,\Hscr)$.
  \end{lemma}
 \begin{proof}
Consider $pq\in \Pi_\Hscr$ and a (simple) $p$--$q$ path $Q$ in $G$. It suffices
to show that
  \begin{equation} \label{eq:ellpQ}
 \ell'(Q)\ge d(pq)-\sum(\lambda(X):\delta X\in\Cscr',\, pq\in \rho_\Hscr X).
 \end{equation}

Let $p,q$ belong to $H'\in\Hscr$ and let the subgraph $Q\cap P$ consist of $k$
components $Y_1,\ldots,Y_k$, occurring in this order in $Q$. We use induction
on $k$.

Suppose that $k\ge 2$. Take vertices $u\in Y_1$ and $v\in Y_2$, and let
$Q':=Q[u,v]$ and $P':=P[u,v]$. Then $\ell'(Q')\ge\ell'(P')$
(by~\refeq{PLshort}). Therefore, replacing in $Q$ the piece $Q'$ by $P'$, we
obtain a $p$--$q$ path $Q''$ in $G$ with $\ell'(Q'')\le \ell'(Q)$ and such that
the number of components of $Q''\cap P$ is less than $k$, and then we apply
induction.

It remains to consider the cases when either (a) $Q\cap P=\emptyset$, or (b)
$Q\cap P$ is nonempty and connected. In case~(a), $Q$ is entirely contained in
one of the subgraphs $G_{P,L}$ and $G':=(G-G_{P,L})\cup P$. If $Q\subset
G_{P,L}$, then~\refeq{ellpQ} follows from~\refeq{pqdist}. And if $Q\subset G'$,
then~\refeq{ellpQ} follows from $\ell'(Q)=\ell(Q)\ge d(pq)$.

In case~(b), if $Q$ is entirely contained in $G_{P,L}$ or in $G'$, then we
argue as in case~(a). So we may assume that $p$ is in $G_{P,L}-P$, ~$q$ is in
$G'-P$. Take a vertex $v$ in $P\cap Q$, and let $Q_1:=Q[p,v]$ and $Q_2:=Q[v,q]$
(which lie in $G_{P,L}$ and $G'$), respectively. By planarity reasons, there is
a unique hole containing both $p,q$, namely, $H$ (where $p$ occurs in $L$, and
$q$ in the other path in $\Lscr_H(st)$). Then $\ell'(Q_1)\ge\tilde d(pv)$
(by~\refeq{pqdist}) and $\ell'(Q_2)=\ell(Q_2)\ge d(vq)$. Also for any $\delta
X\in\Cscr'$, neither $q$ nor $v$ is in $X$. This implies that $pq\in\rho_\Hscr
X$ if and only if $pv\in \rho_{\Hscr'}X$. Hence $a:=\sum(\lambda(X)\colon
\delta X\in \Cscr',\, pq\in\rho_\Hscr X)$ is equal to $b:=\sum(\lambda(X)\colon
\delta X\in \Cscr',\, pv\in\rho_{\Hscr'}X)$, and we have
  \begin{multline*}
\ell'(Q)=\ell'(Q_1)+\ell(Q_2)\ge \tilde d(pv)+d(vq) \\
  =d(pv)-b+d(vq)=d(pv)+d(vq)-a\ge d(pq)-a,
  \end{multline*}
implying~\refeq{ellpQ}.
 \end{proof}

Applying the above procedure, step by step, to the excessive pairs of types 0
and 1, we get rid of all such pairs. Each pair in $\Pi_\Hscr$ is treated at
most once (in view of~\refeq{zero-zero}), and therefore the whole process,
called \emph{Reduction~I}, takes $O(|V|^2)$ iterations and is implemented in
strongly polynomial time (relying on the complexity of the algorithm
in~\cite{kar90}).

So we may further assume that
\begin{description}
\item{(C4):} ~$(G,\ell,\Hscr)$ has no excessive pairs $(P,L)$ with $\tau(P,L)\le
1$; therefore, when $|\Hscr|\le 4$, for any $H\in\Hscr$ and $s,t\in V_H$, at
least one of the two $s$--$t$ paths $L,L'$ in $\bd(H)$ is shortest:
$d(st)=\min\{\ell(L),\ell(L')\}$
  \end{description}
(since at least one of the pairs $(P,L)$ and $(P,L')$ has type $\le 1$). One
more useful observation for $|\Hscr|=3$ is:
\begin{numitem1} \label{eq:diametr}
if $\Hscr=\{H,H',H''\}$ and $P$ is an $H$-geodesic with ends $s,t$ separating
the holes $H'$ and $H''$, then both paths in $\Lscr_H(st)$ are shortest.
  \end{numitem1}


\section{Elimination of lenses without holes} \label{sec:0-lenses}

Consider distinct holes $H,H'\in\Hscr$, an $H$-geodesic $P$, and an
$H'$-geodesic $P'$. Suppose that $P,P'$ have common vertices $x,y$ and let $Q$
(resp. $Q'$) be the subpath of $P$ (resp. $P'$) between $x$ and $y$. Suppose
that the interiors of $Q$ and $Q'$ are disjoint. Then the fact that both $P,P'$
are shortest implies $P\cap Q'=P'\cap Q=\{x,y\}$.
 \medskip

\noindent\textbf{Definition.} ~We denote the (closed) region of the sphere
bounded by $Q\cup Q'$ and containing neither $H$ nor $H'$ by $\Omega(Q,Q')$ and
call it a \emph{lens} for $P,P'$ with ends $x,y$. If, in addition,
$\Omega(Q,Q')$ contains no hole, it is called a \emph{0-lens}. (See the
picture.)

\vspace{-0cm}
\begin{center}
\includegraphics[scale=0.9]{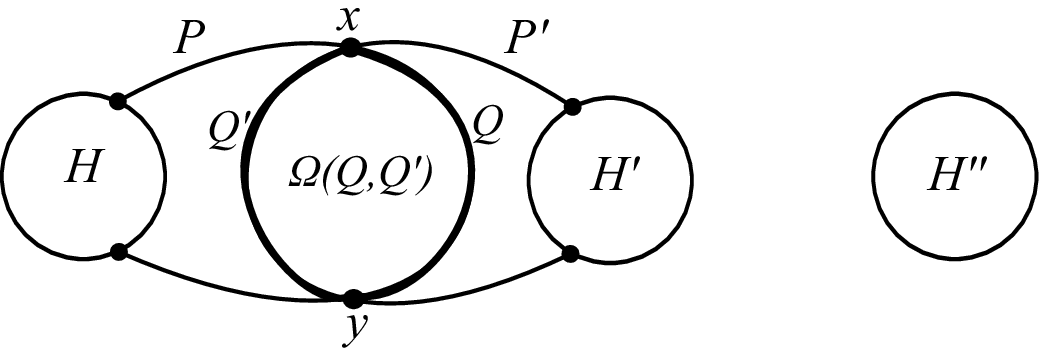}
\end{center}
\vspace{-0.3cm}

In what follows, for a vertex $v$ and an edge $e$ of a path $P$, we may
liberally write $v\in P$ for $v\in V_P$ and $e\in P$ for $e\in E_P$.

It turns out that conditions (C2)--(C3) provide the following nice property.
\begin{prop} \label{pr:0-lens}
There exists no 0-lens at all.
  \end{prop}
  \begin{proof}
Suppose, for a contradiction, that a 0-lens $\Omega(Q,Q')$ does exist, and let
$H,H',P,P',x,y, Q,Q'$ be as above. Let $G_{Q,Q'}=(V_{Q,Q'},E_{Q,Q'})$ be the
subgraph of $G$ lying in $\Omega:=\Omega(Q,Q')$. We rely on the following
  \medskip

\noindent\textbf{Claim} ~\emph{In the graph $G_{Q,Q'}$, each edge $e$ is
contained in a shortest $x$--$y$ path, and similarly for any pair of vertices
in a face of $G_{Q,Q'}$ in $\Omega$.}
 \medskip

 \noindent\underline{Proof of Claim.}
~One may assume that $e$ is an inner edge of $G_{Q,Q'}$ (i.e., not on $Q\cup
Q'$). Take an $\Hscr$-geodesic $L$ containing $e$. Using the fact that $\Omega$
has no hole, one can realize that $L$ crosses at least twice some of $P,P'$.
Moreover, there are two vertices $u,v$ of $L$ such that $e$ belongs to
$L':=L[u,v]$ and at least one of the following takes place: both $u,v$ are in
$P$; both $u,v$ are in $P'$. Assume that $u,v$ are chosen so that $L'$ is
minimal under this property and let for definiteness $u,v\in P$. Let
$R:=P[u,v]$.

Let $\tilde L$ be the path obtained from $P$ by replacing its part $R$ by $L'$;
this is again a geodesic containing $e$. Moreover, the minimal choice of $L'$
implies that at least one of the vertices $u$ and $v$, say, $u$, belongs to
$Q$, while the other, $v$, either (a) belongs to $Q$ as well, or (b) is not in
$\Omega$.

In case~(a), we may assume that $L'$ is entirely contained in $\Omega$. Then
replacing in $Q$ the part between $u$ and $v$ by $L'$, we just obtain the
desired $x$--$y$ path containing $e$.

In case~(b), $L'$ meets $Q'$ at a vertex $v'$ such that $L'[u,v']$ lies in
$\Omega$ and contains $e$. Let $s,t$ be the ends of $\tilde L$ (and $P$); we
may assume that $s,u,v',v,t$ occur in this order in $\tilde L$. Since $u\in Q$,
the subpath $\tilde L[s,u]$ ($=P[s,u]$) passes one of the ends $x,y$ of the
lens $\Omega$, say, $x$; see the picture below. Then the desired shortest
$x$--$y$ path in $G_{Q,Q'}$ containing $e$ is obtained by concatenating
$Q[x,u]$, $L'[u,v']$ and $Q'[v',y]$.

\vspace{-0cm}
\begin{center}
\includegraphics{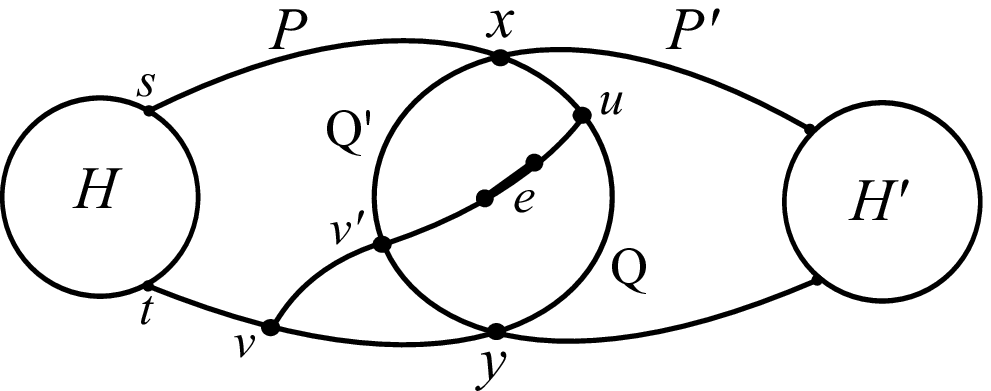}
\end{center}
\vspace{-0cm}

The assertion for a pair of vertices $w,z$ in a face $F$ of $G_{Q,Q'}$ in
$\Omega$ is proved in a similar way. More precisely, take an $\Hscr$-geodesic
$D$ passing $w$ and $z$, existing by~(C2), and let $D':=D[w,z]$. Making, if
needed, appropriate exchange operations involving $P,D'$ and/or $P',D'$, one
can ``improve'' $D'$ so as to get it entirely contained in $\Omega$ (keeping
$w,z$). Now we argue as above, with $D'$ in place of $e$. \hfill\qed
 \medskip

By the Claim, $G_{Q,Q'}$ is the union of shortest $x$--$y$ paths; therefore,
one can direct the edges of $G_{Q,Q'}$ so that each shortest $x$--$y$ path
turns into a directed $x$--$y$ path, and vice versa. Then each face $F$ of
$G_{Q,Q'}$ in $\Omega$ has two vertices $x_F$ and $y_F$ such that $\bd(F)$ is
formed by two $x_F$--$y_F$ paths $A$ and $B$, which are extended to shortest
$x$--$y$ paths $C:=P'_F\cdot A\cdot P''_F$ and $D:=P'_F\cdot B\cdot P''_F$
(where $P'_F,P''_F$ are shortest $x$--$x_F$ and $y_F$--$y$ paths in $G_{Q,Q'}$,
respectively).

Suppose that $A$ has an intermediate vertex $u$ and $B$ has an intermediate
vertex $v$. By the Claim, $u$ and $v$ belong to a shortest $x$--$y$ path $L$ in
$G_{Q,Q'}$; let for definiteness $x,u,v,y$ occur in this order in $L$. By the
planarity, $L':=L[u,v]$ must intersect either (a) the path $P'_F$, or (b) the
path $P''_F$. The graph $G_{Q,Q'}$, being directed as indicated above, is
acyclic. But in case~(a), the subgraph $C\cup L'$ has a directed cycle, and in
case~(b), so does the subgraph $D\cup L'$; a contradiction.

Thus, either $A$ or $B$ has no intermediate vertex, i.e., has only one edge
$e$. Since $\ell(A)=\ell(B)$, $e$ is dominating in $F$. This contradicts~(C3),
and the result follows.
 \end{proof}

One consequence of the non-existence of 0-lenses that will be used later is as
follows.

 \begin{numitem1} \label{eq:barrier}
Let $(P_i,L_i),\ldots,(P_k,L_k)$ be normal pairs of type 0 for a hole $H$, and
let $\Omega_i:=\Omega(P_i,L_i)\cup H$. Then for each hole $H'\ne H$, no
$H'$-geodesic has a vertex in $\inter(\Omega_1\cup\cdots\cup\Omega_k)$.
  \end{numitem1}

Indeed, such an $H'$-geodesic would create a 0-lens with some $P_i$.


\section{Necklaces} \label{sec:neclac}

In this section we further simplify the graph $(G,\ell)$ by handling one more
sort of reducible cuts. In fact, our description in the previous sections was
applicable to an arbitrary number of holes. This and the next sections will be
devoted to the three-hole case only (though some ingredients are valid for
$|\Hscr|>3$ as well).

Fix a hole $H$ and denote the $\ell$-length of $\bd(H)$ by $\sigma=\sigma_H$.
To simplify our description technically, we insert (for a while) extra
terminals in the boundary $\bd(H)$ to make it central symmetric. More
precisely, for each $s\in V_H$, when $s$ does not have the antipodal terminal
in $\bd(H)$, we add such a vertex by splitting the corresponding edge $pq\in
E_H$ into two edges $pt$ and $tq$ whose lengths are such that
$\ell(pt)+\ell(tq)=\ell(pq)$ and $d(st)=d(sp)+\ell(pt)=d(sq)+\ell(qt)=\sigma/2$
(using $d(sp)+\ell(pq)+d(qs)=\sigma$, cf.~(C4)). This modification does not
affect the problem, and we keep the previous notation $(G,\ell,\Hscr)$. (It
increases the number of vertices and, possibly, creates non-tight pair of
vertices involving $t$, violating~(C3); but this will not be important for us.)

Let $s_1,s_2,\ldots,s_{2n}=s_0$ be the sequence of vertices of the (modified)
boundary cycle $\bd(H)$ in the clockwise order around $\inter(H)$. For each
$s_i$, its antipodal vertex $s_{i+n}$ is also denoted as $t_i$ (hereinafter the
indices are taken modulo $2n$). A path of the form $s_is_{i+1} \cdots s_j$ is
denoted by $L(s_is_j)$; then $L(s_js_i)$ is the path ``complementary'' to
$L(s_is_j)$ in $\bd(H)$. When vertices $s_{i(1)},s_{i(2)},\ldots,s_{i(k)}$
(admitting $i(j)=i(j+1)$) follow in this order cyclically, making at most one
turn, we write $s_{i(1)}\to s_{i(2)}\to \cdots \to s_{i(k)}$.

When $L:=L(st)$ is shortest, we denote by $\Pscr(st)$ the set of shortest
$s$--$t$ paths $P$ such that $\tau(P,L)=0$. For $P,P'\in\Pscr(st)$, let
$P\wedge P'$ (resp. $P\vee P'$) be the $s$--$t$ path which together with $L$
forms the boundary of $\Omega(P,L)\cap \,\Omega(P',L)$ (resp. $\Omega(P,L)\cup
\,\Omega(P',L)$). Then $\ell(P\wedge P') +\ell(P\vee P')=\ell(P)+\ell(P')$,
implying that both $P\wedge P',\, P\vee P'$ are shortest as well. Hence
$(\Pscr(st),\wedge,\vee)$ is a lattice with the minimal element $L$. The
maximal (most remote from $L$) element of $\Pscr(st)$ is denoted by $D(st)$;
then $\Omega(P,L)\subseteq\Omega(D(st),L)$ for any $P\in \Pscr(st)$. (Note that
$D(st)$ can be extracted from the subgraph of shortest $s$--$t$ paths in
strongly polynomial time.)

We are going to examine an interrelation of paths for two antipodal pairs
$s_i,t_i$ and $s_j,t_j$ with $s_i\to s_j\to t_i\to t_j$. For brevity we write
$L_i,\Pscr_i, D_i,\Omega_i$ for $L(s_it_i),\, \Pscr(s_it_i),\, D(s_it_i),\,
\Omega(D(s_it_i),L(s_it_i))$, respectively, and similarly for $j$. Also we
abbreviate $s:=s_i$, $t:=t_i$, $s':=s_j$, and $t':=t_j$.

Consider paths $M\in\Pscr_i$ and $N\in \Pscr_j$. Then $M\cap N\ne\emptyset$;
let $x$ and $y$ be the first and last vertices of $M$ occurring in $M\cap N$.
Since $M,N$ are shortest, $\ell(M[x,y])=\ell(N[x,y])$. Two cases are possible:

\begin{description}
\item[\rm\emph{Case 1}]: ~either $x=y$ or $x$ precedes $y$ in $N$;
\item[\rm\emph{Case 2}]: $y$ precedes $x$ in $N$.
 \end{description}

We first consider Case 1. Represent $M$ as $M_1\cdot M_2\cdot M_3$, and $N$ as
$N_1\cdot N_2\cdot N_3$, where $M_2:=M[x,y]$ and $N_2:=N[x,y]$. One can see
that exchanging the parts $M_2$ and $N_2$ in $M$ and $N$, we again obtain paths
of type 0, one belonging to $\Pscr_i$, and the other to $\Pscr_j$. To slightly
simplify our considerations, we will assume that $M_2=N_2$, denoting this path
as $\hat M$. (Note that if $M=D_i$ and $N=D_j$, then $M_2=N_2$ follows from the
maximality of $D_i,D_j$.) Form the $s$--$t'$ path $P$ and the $s'$--$t$ path
$P'$ as
   $$
   P:=M_1\cdot\hat M\cdot N_3\quad \mbox{and}\quad P':=N_1\cdot\hat M\cdot M_3,
   $$
and let $\Omega:=\Omega(P,L(st'))$ and $\Omega':=\Omega(P',L(s't))$. Clearly
both regions $\Omega',\Omega$ contain no hole; also $\Omega\subseteq
\Omega_i\cup\Omega_j$ and $\Omega'\subseteq \Omega_i\cap\Omega_j$. See the left
fragment of the picture (where $H$ is the outer face of $G$).

 \vspace{-0cm}
\begin{center}
\includegraphics[scale=1.0]{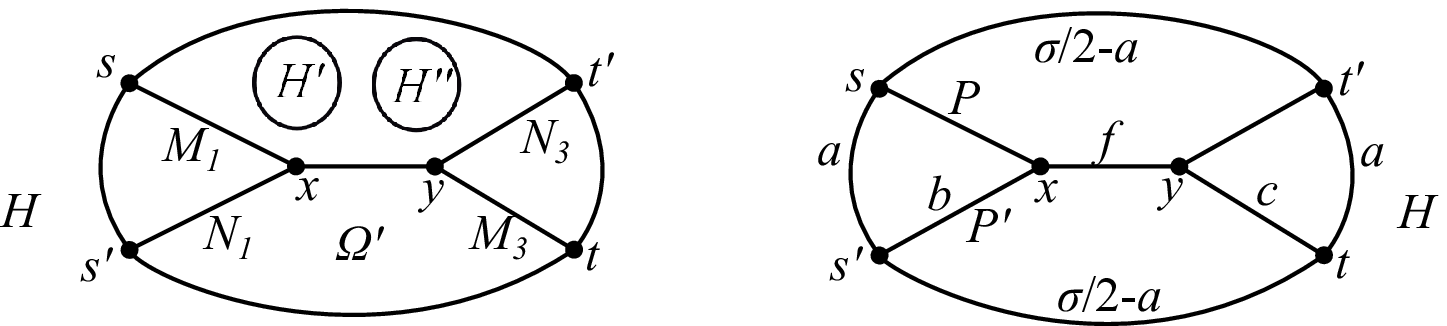}
\end{center}
\vspace{-0.3cm}

Define
  $$
a:=d(ss')\;\; (=d(tt')),\quad b:=d(s'x),\quad c:=d(yt),\quad\mbox{and}\quad
f:=d(xy).
  $$
Since the paths $L_i,L_j,M,N$ are shortest and have the same length $\sigma/2$,
we have
 \begin{equation} \label{eq:dsix}
\ell(P')=b+f+c\ge d(s't)=\sigma/2-a\;\; \mbox{and}\;\;
\ell(P)=\sigma-\ell(P')\le\sigma/2+a.
  \end{equation}

We distinguish between two subcases:

\begin{description}
\item[\rm\emph{Subcase 1a}]: ~the path $P'$ is shortest: $\ell(P')=d(s't)$;
\item[\rm\emph{Subcase 1b}]: ~$\ell(P')>d(s't)$.
\end{description}

If Subcase 1b happens, we devise a certain collection of reducible cuts and
make a good reduction, aiming to obtain a situation as in Subcase~1a. For this
purpose, we apply the algorithm of~\cite{kar90} to solve the auxiliary one-hole
PMP with $(G,\ell,\{H\})$, i.e., we handle the same $G$ and $\ell$ but regard
$H',H''$ as inner faces (see the right fragment of the above picture). It finds
a packing $\Cscr$ of (simple) cuts $\delta X$ with integer weights
$\lambda(X)>0$ realizing the distances on $\Pi_{H}$.

Let $\eps:=\ell(P')-d(s't)$. Then (cf.~\refeq{dsix})
\begin{equation}  \label{eq:ellPp}
\ell(P')=\sigma/2-a+\eps \quad\mbox{and}\quad \ell(P)=\sigma-\ell(P')=\sigma/2
+a-\eps.
  \end{equation}

Since the path $M$ is shortest and connects antipodal terminals, each cut
$\delta X\in\Cscr$ meets $M$ exactly once, and similarly for $N$. Let $\Cscr'$
be the set of cuts $\delta X\in\Cscr$ meeting $L(s't)$. Then each $\delta
X\in\Cscr'$ meets $L(t's)$ as well, whereas each $\delta X\in\Cscr-\Cscr'$
meets one edge in each of $L(ss')$ and $L(tt')$. Partition $\Cscr-\Cscr'$ as
$\Cscr_1\cup\Cscr_2$, where $\Cscr_1$ is formed by the cuts not meeting $P$,
and $\Cscr_2$ is the rest (consisting of the cuts $\delta X$ with $|\delta
X\cap M_1|=|\delta X\cap N_3|=1$). Let $h',h_1,h_2$ be the sums of values
$\lambda(X)$ over the cuts $\delta X$ in $\Cscr',\Cscr_1,\Cscr_2$,
respectively. The cuts in $\Cscr$ must saturate the shortest paths
$L_i,L_j,M,N$, and therefore they saturate $P$ and $P'$. So $L(s't)$ and
$L(t's)$ are saturated by the cuts of $\Cscr'$, ~$P$ by the cuts of
$C'\cup\Cscr_2$, and $L(ss')\cup L(tt')$ by the cuts of $\Cscr_1\cup\Cscr_2$.
Then (in view of~\refeq{ellPp})
  \begin{equation} \label{eq:PPprim}
 h'=\sigma/2-a,\quad h_2=(\ell(P)-h')/2=a-\eps/2\quad\mbox{and}\quad
h_1=a-h_2=\eps/2.
  \end{equation}

For each $\delta X\in\Cscr_1$, since $\delta X$ do not meet $P$, we may assume
that  $X\subset\Omega-P$. Let $\lambda_1$ be the restriction of $\lambda$ to
$\Cscr_1$. We assert the following
  \begin{lemma} \label{lm:Cscr1}
$(\Cscr_1,\lambda_1)$ is reducible for $(G,\ell,\Hscr)$; in other words, for
the reduced length $\ell':=\ell-\sum(\lambda(X)\chi^{\delta X}\colon \delta
X\in\Cscr_1)$, any $p$--$q$ path $Q$ in $G$ with $pq\in\Pi_\Hscr$ satisfies
 \begin{equation} \label{eq:QC1}
\ell'(Q)\ge d(pq)-\sum(\lambda(X)| \{pq\}\cap\rho_\Hscr X| \colon \delta X\in
\Cscr_1).
  \end{equation}
  \end{lemma}
 \begin{proof}
If $Q$ is a $p$--$q$ path with $p,q\in V_H$, then~\refeq{QC1} immediately
follows from the reducibility of $(\Cscr_1,\lambda_1)$ for $(G,\ell,\{H\})$. So
assume that $pq\in\Pi_{H'}\cup \Pi_{H''}$. Then both $p,q$ are not in
$\Omega-P$.

If $Q$ does not meet $\Omega-P$, then we have $\ell'(Q)=\ell(Q)$ (in view of
$X\subset \Omega-P$ for each $\delta X\in\Cscr_1$), and~\refeq{QC1} is trivial.
Suppose that $Q\cap(\Omega-P)\ne\emptyset$. Then $Q$ crosses at least one of
$M,N$. Let for definiteness $Q$ meets $\Omega(M,L_i)-M$, and take a maximal
subpath $R$ of $Q$ such that $R\subset\Omega(M,L_i)$ and $R\not\subset M$. Let
$u,v$ be the endvertices of $R$, and let $M':=M[u,v]$. The fact that $M$ is
$\ell'$-shortest implies that $\ell'(M')\le \ell'(R)$. Then, replacing in $Q$
the part $R$ by $M'$, we obtain a $p$--$q$ path $Q'$ with $\ell'(Q')\le
\ell'(Q)$. If $Q'$ still meets $\Omega-P$, we repeat the procedure (treating
the pair $(Q',M)$ or $(Q',N)$), and so on. Eventually, we obtain a $p$--$q$
path $\tilde Q$ such that $\ell'(\tilde Q)\le\ell'(Q)$ and $\tilde Q\cap
(\Omega-P)=\emptyset$, yielding~\refeq{QC1} for $Q$.
 \end{proof}

From~\refeq{ellPp},\refeq{PPprim} it follows that for the updated length
function,
  \begin{numitem1} \label{eq:updatell}
the path $P'$ becomes shortest (of length $\sigma/2-a$), and the lengths of $P$
and $L(st')$ become the same.
  \end{numitem1}
Thus, the above procedure turns Subcase~1b into Subcase~1a, as required.
 \smallskip

Next we consider Case~2. Let $x=v_0,v_1,\ldots, v_k=y$ be the common vertices
of $M$ and $N$, in this order in $M$ and, accordingly, in the reverse order in
$N$ (taking into account that $M,N$ are shortest). For $p=1,\ldots,k$, let
$\omega_p$ be the region bounded by the subpaths $M(p):=M[v_{p-1},v_p]$ and
$N(p):=N[v_p,v_{p-1}]$ and not contained $H$. Then either (a) $\omega_p$ lies
in $\Omega(M,L_i)\cap\Omega(N,L_j)$ (in particular, $M(p)$ may coincide with
the reverse path $N(p)^{-1}$ to $N(p)$), or (b) $\inter(\omega_p)$ is nonempty
and lies in the complement of $\Omega(M,L_i)\cup\Omega(N,L_j)$. To slightly
simplify our considerations, we exclude the cases when $\inter(\omega_p)$ is
nonempty and contains no hole (which will lead to no loss of generality).
Namely, for each $p$ where such a situation happens, we perturb $N$, by
replacing its part $N(p)$ by $M(p)^{-1}$. Clearly the updated $s'$--$t$ path
$N$ is again shortest and of type 0.

So we will further assume that for each $p$ where $M(p)\ne N(p)^{-1}$, the
region $\omega(p)$ contains one or two holes among $H',H''$; in this case we
say that $\omega_p$ is \emph{essential}. We come to four subcases of Case~2.

\begin{description}
\item[\rm\emph{Subcase 2a}]: ~No region $\omega_p$ is essential; equivalently,
$M[x,y]=N^{-1}[y,x]$.
\item[\rm\emph{Subcase 2b}]: ~Only one $\omega_p$ is essential and it contains
exactly one hole, say, $H'$.
\item[\rm\emph{Subcase 2c}]: ~Two $\omega_p,\omega_q$ are essential
(each containing one hole).
\item[\rm\emph{Subcase 2d}]: ~One $\omega_p$ is essential and it contains both
$H',H''$.
  \end{description}

We first handle (simultaneously) Subcases~2a and~2b; they are illustrated in
the left and right fragments of the picture, respectively.

 \vspace{-0cm}
\begin{center}
\includegraphics[scale=1.0]{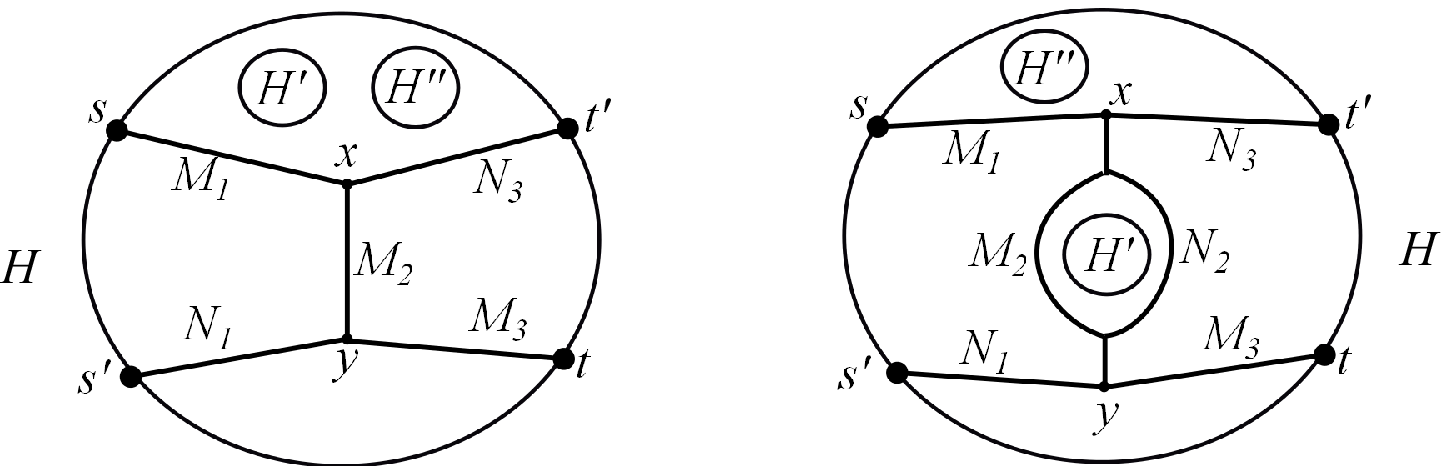}
\end{center}
\vspace{-0cm}

We are going to reduce these subcases to a situation as in Subcase~1a. As
before, let $a:=d(ss')$. Represent $M$ as $M_1\cdot M_2\cdot M_3$, and $N$ as
$N_1\cdot N_2\cdot N_3$, where $M_2:=M[x,y]$ and $N_2:=N[y,x]$ (then
$M_2=N_2^{-1}$ in Subcase~2a). Form the $s$--$t'$ path $P:=M_1\cdot N_3$ and
the $s'$--$t$ path $P':=N_1\cdot M_3$ and define
   $$
   f:=\ell(M_2)\;\; (=\ell(N_2))\quad \mbox{and}\quad \eps:=\ell(P')-d(s't).
   $$
Then $\ell(P')=d(s't)+\eps=\sigma/2-a+\eps$. Since $\ell(M)=\ell(N)=\sigma/2$,
we have
  \begin{equation} \label{eq:PMNPp}
  \ell(P)=\ell(M)+\ell(N)-\ell(P')-\ell(M_2)-\ell(N_2)=\sigma/2+a-\eps-2f.
  \end{equation}

To make the desired reduction, we use the algorithm of~\cite{kar90} to solve
the auxiliary one- or two-hole PMP with $(G,\ell,\tilde\Hscr)$, where
$\tilde\Hscr:=\{H\}$ in Subcase~2a, and $\tilde\Hscr:=\{H,H'\}$ in Subcase~2b.
Let $(\Cscr,\lambda)$ be an integer solution to it. We extract from $\Cscr$ the
set $\Cscr'$ of cuts meeting $L(s't)$ (and its opposite path $L(t's)$) and the
set $\Cscr''$ of cuts meeting $L(ss')$ (and $L(tt')$). Then
 \begin{equation} \label{eq:hp-hpp}
 h':=\sum(\lambda(X)\colon \delta X\in\Cscr')=\sigma/2-a \quad \mbox{and} \quad
  h'':=\sum(\lambda(X)\colon \delta X\in\Cscr'')=a.
  \end{equation}

Partition $\Cscr''$ as $\Cscr_1\cup\Cscr_2$, where $\Cscr_1$ consists of the
cuts not meeting $P$, and accordingly define $h_i:=\sum(\lambda(X)\colon \delta
X\in\Cscr_i)$, $i=1,2$. Each cut $\delta X\in\Cscr_2$ meets $P$ twice (since
$\delta X\cap P\ne\emptyset$ implies $|\delta X\cap M_1|= |\delta X\cap
N_3|=1$). Therefore, $\ell(P)\ge h'+2h_2$, and using~\refeq{PMNPp}
and~\refeq{hp-hpp}, we have
  $$
  2h_2\le\ell(P)-h'=(\sigma/2+a-\eps-2f)-(\sigma/2-a)=2a-2f-\eps.
  $$
This and $h_1+h_2=a$ imply
  \begin{equation} \label{eq:h1feps}
  h_1=f+\eps/2.
  \end{equation}

The following assertion is similar to Lemma~\ref{lm:Cscr1}.
  \begin{lemma} \label{lm:Cases2a2b}
In Subcases~2a,2b, ~$(\Cscr_1,\lambda_1)$ is reducible for $(G,\ell,\Hscr)$,
where $\lambda_1:=\lambda\rest{\Cscr_1}$.
  \end{lemma}
 \begin{proof}
Let $Q$ be a $p$--$q$ path in $G$ with $pq\in\Pi_\Hscr$. We have to show
relation~\refeq{QC1} for $Q$ and $\ell':=\ell-\sum(\lambda(X)\chi^{\delta
X}\colon \delta X\in\Cscr_1)$. This is done in a way similar to the proof of
Lemma~\ref{lm:Cscr1}. More precisely, if $pq\in\Pi_{\tilde\Hscr}$,
then~\refeq{QC1} is immediate from the reducibility of $(\Cscr_1,\lambda_1)$
for $(G,\ell,\tilde \Hscr)$. And if $pq\in\Pi_{\Hscr-\tilde\Hscr}$, then both
$p,q$ are not in $\Omega-P$, where $\Omega:=\Omega(P,L(st'))$. Making, if
needed, the corresponding replacements in $Q$ using pieces of $M$ or $N$ (like
in the proof of Lemma~\ref{lm:Cscr1}), we obtain a $p$--$q$ path $Q'$ disjoint
from $\Omega-P$ and such that $\ell'(Q')\le\ell'(Q)$. Then
$\ell'(Q')=\ell(Q')$, implying~\refeq{QC1} for $Q$.
 \end{proof}

Note that since each cut $\delta X\in\Cscr_1$ does not meet $P$ and $M,N$ are
shortest, either $|\delta X\cap M_2|=1$ or $|\delta X\cap N_1|=|\delta X\cap
M_3|=1$. Then~\refeq{h1feps} implies that $(\Cscr_1,\lambda_1)$ saturates $M_2$
and uses $\eps$ units of the $\ell$-length of $P'$. It follows that for the
updated lengths,
  \begin{numitem1} \label{eq:PpM2N2}
the path $P'$ becomes shortest, and the lengths of $M_2$ and $N_2$ become zero.
  \end{numitem1}
In other words, contracting the edges with zero length, we obtain a situation
as in Subcase~1a, as required. (Note that the hole $H'$ vanishes if $|\tilde
\Hscr|=2$.)
 \smallskip

In Subcase~2c we act in a similar fashion. Suppose that $H'\subset\omega_p$,
$H''\subset \omega_q$, and $q<p$. Let $z:=v_q$. Then $H''$ is located between
$M[x,z]$ and $N[z,x]$, whereas $H'$ between $M[z,y]$ and $N[y,z]$; see the left
fragment of the picture.

 \vspace{-0cm}
\begin{center}
\includegraphics[scale=1.0]{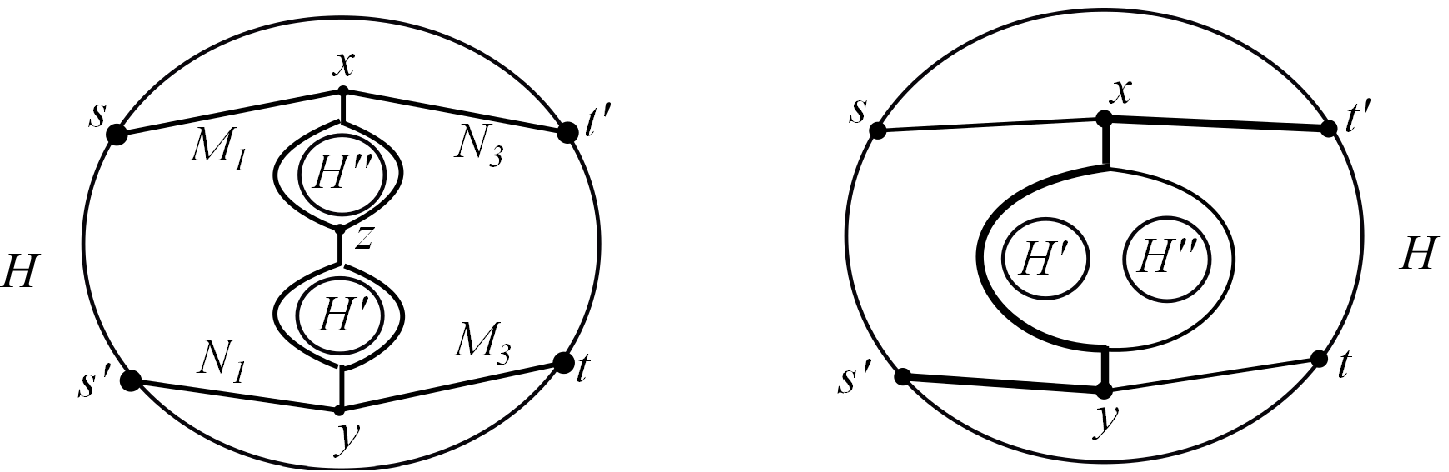}
\end{center}
\vspace{-0.2cm}

Let $P':=N_1\cdot M_3$, $\eps:=\ell(P')-d(s't)$ and $f:=\ell(M[z,y])$
($=\ell(N[y,z])$). Find an integer solution $(\Cscr,\lambda)$ to the two-hole
PMP with $(G,\ell,\tilde\Hscr)$, where $\tilde\Hscr:=\{H,H'\}$. Extract from
$\Cscr$ the set $\Cscr''$ of cuts meeting $L(ss')$ and partition $\Cscr''$ as
$\Cscr_1\cup\Cscr_2$, where $\Cscr_1$ is formed by the cuts not meeting
$M[s,z]$. Let $\lambda_1$ be the restriction of $\lambda$ to $\Cscr_1$. Arguing
as in the previous case, one can conclude that $(\Cscr_1,\lambda_1)$ is
reducible for $(G,\ell,\Hscr)$ and that $\sum(\lambda(X)\colon \delta
X\in\Cscr_1)=f+\eps/2$. Then for the reduced length function, the path $P'$
becomes shortest, and the lengths of the paths $M[z,y]$ and $N[y,z]$ become
zero. Then the hole $H'$ vanishes and we obtain PMP with two holes.
 \smallskip

It remains to consider Subcase~2d. We distinguish between two possibilities for
$s=s_i$ and $s'=s_j$.
  \smallskip

I. Suppose that $0<j-i<n$, i.e., $s_i$ and $s_j$ are not antipodal. As before,
we represent $M$ as $M_1\cdot M_2\cdot M_3$, and $N$ as $N_1\cdot N_2\cdot
N_3$, where $M_2:=M[x,y]$ and $N_2:=N[y,x]$. For the $s$--$s'$ path
$R:=M_1\cdot M_2\cdot N_1^{-1}$ and the $t$--$t'$ path $R':=M_3^{-1}\cdot N_2
\cdot N_3$, and for $a:=\ell(L(ss'))$, we have
  $$
  \ell(R)+\ell(R')=\ell(M)+\ell(N)=\sigma>2a.
  $$
Let for definiteness $\ell(R)\ge\ell(R')$; then $\ell(R)>a$. Instead of $N$, we
consider the shortest $t'$--$s'$ path $N':=N_3^{-1}\cdot M_2\cdot N_1^{-1}$
(i.e., we change $j$ to $j+n$), and handle the pair $N',M$ rather than $M,N$.
See the right fragment of the above picture where $N'$ is drawn in bold. The
path $N'$ is of type 0 and the pair $N',M$ is as in Subcase~1b (since
$\ell(R)>\ell(L(ss'))$). We apply a good reduction in this subcase (which
decreases the lengths of $L(s't)$ and $L(t's)$ and makes $R$ be shortest).
 \smallskip

II. Now suppose that $j=i+n$; then $s=t'=x$ and $s'=t=y$. Let $z:=v_{p-1}$ and
$u:=v_p$, i.e., the region $\omega_p$ containing $H'$ and $H''$ is bounded by
the paths $M[z,u]$ and $N[u,z]$. Suppose that $u\ne y$. Then we make a good
reduction, aiming to decrease the length of the path $M':=M[u,y]$
($=N^{-1}[u,y]$) to zero (and then to contract $M'$). For this purpose, we find
a solution $(\Cscr,\lambda)$ to the auxiliary one-hole problem PMP with
$(G,\ell,\{H\})$ and extract from $\Cscr$ the set $\Cscr_1$ of cuts meeting
$M'$. Since $M,N$ are shortest paths connecting the antipodal terminals $x,y$
and both holes $H',H''$ are disposed between $M[z,u]$ and $N[u,z]$, one can
conclude that $(\Cscr_1,\lambda_1)$ is reducible for $(G,\ell,\Hscr)$ and
saturates $M'$. Then, after the reduction using $(\Cscr_1,\lambda_1)$ followed
by the corresponding contractions, $u$ and $y$ become merged into one vertex.

If $x\ne z$, we handle the subpath $M[x,z]$ in a similar way.
  \medskip

We apply the good reductions by cuts described above to all pairs
$\{i,j\}\subset\{1,\ldots,2n\}$ for the hole $H$ (with $\bd(H)$ extended to be
central symmetric), and then treat the other holes $H',H''$ in a similar way,
referring to the whole process as \emph{Procedure~II}. Every time we take as
$M,N$ the most remote paths $D_i\in\Pscr_i$ and $D_j\in\Pscr_j$. Summing up the
above results, we can conclude with the following

\begin{prop} \label{pr:necklaceH}
Procedure II takes strongly polynomial time and results in $(G,\ell,\Hscr)$,
where $\ell$ is cyclically even and $|\Hscr|\le 3$, so that: for each
$H\in\Hscr$ and antipodal pairs $\{s_i,t_i\}$ and $\{s_j,t_j\}$ with $s_i\to
s_j\to t_i\to t_j$ in $\bd(H)$, the paths $M:=D(s_it_i)$ and $N:=D(s_jt_j)$ are
subject to Subcase~1a when $s_i\ne t_j$, and subject to Subcase~2d with $M\cap
N=\{s_i,s_j\}$ when $s_i= t_j$.
  \end{prop}

Removing the extra terminals that were added before the procedure, we obtain a
similar result in terms of the original graph. More precisely, for $H\in\Hscr$,
let $s_1,s_2,\ldots,s_m=s_0$ be the vertices of (the original) $\bd(H)$ in the
clockwise order, and define $\Lscr^{\max}_H$ to be the set of inclusion-wise
\emph{maximal} shortest paths $s_is_{i+1}\cdots s_{i'}$ in $\bd(H)$ (taking
indices modulo $m$). For such a path, we denote $s_{i'}$ by $t_i$. Also we
denote the set of indices $i$ for which $\Lscr^{\max}_H$ has a path starting
with $s_i$ by $I=I_H$. Like the central symmetric case, for $i\in I$, we denote
$L(s_it_i)$ by $L_i$ and write $\Pscr_i,D_i,\Omega_i$ for the set of shortest
$s_i$--$t_i$ paths of type 0, the most remote path in $\Pscr_i$, and the region
$\Omega(D_i,L_i)$, respectively. Then (after removing the extra terminals) the
following holds:
 \begin{numitem1} \label{eq:casei}
for $H\in\Hscr$ and $i,j\in I_H$ with $s_i\to s_j\to t_i\to t_j$: (a) if
$s_i=t_j$ (and therefore $s_j=t_i$), then $D_i\cap D_j=\{s_i,s_j\}$; and (b) if
$s_i\ne t_j$, then for any common vertex $v$ of $D_i$ and $D_j$, the
$s_j$--$t_i$ path $D_j[s_j,v]\cdot D_i[v,t_i]$ is shortest.
  \end{numitem1}

For $H\in\Hscr$ fixed and $i\in I_H$, let $G_i=(V_i,E_i)$ be the directed graph
that is the union of shortest paths $P\in\Pscr(s_it_i)$, each $P$ being
directed from $s_i$ to $t_i$. We will take an advantage from the following
 \begin{prop} \label{pr:neclace}
{\rm(i)} The directions of edges in all graphs $G_i$ are agreeable.

{\rm(ii)} In the graph $N_H=(W,U):=\cup(G_i\colon i\in I)$, for each simple
directed cycle $C$, the region bounded by $C$ and containing $H$ contains no
other hole.

{\rm(iii)} $N_H$ admits a \emph{cyclic potential} $\pi:W\to \Zset$, which means
that for each edge $e=(u,v)\in U$, $\ell(e)=\pi(v)-\pi(u)$ if
$\pi(u)\le\pi(v)$, and $\ell(e)=\pi(v)-\pi(u)+\sigma_H$ otherwise.
  \end{prop}
 \begin{proof}
For $k=1,\ldots,m=|V_H|$, define $\pi(s_k):=\ell(L(s_1s_k))$. Then each edge
$e=s_ks_{k+1}$ of $\bd(H)$ satisfies $\ell(e)=\pi(s_{k+1})-\pi(s_k)$, taking
indices modulo $m$ and taking lengths/potentials modulo $\sigma=\sigma_H$.
Hence $\ell(L(s_it_i))= \pi(t_i)-\pi(s_i)$ for each $i\in I$.

In order to extend $\pi$ to the other vertices of $N_H$ we first introduce, for
each $i\in I$, its own potential $\pi_i:= V_i\to \Zset$ as
  $$
 \pi_i(v):=\pi_i(s_i)+\ell(P),
  $$
where $v\in V_i$ and $P$ is a directed $s_i$--$v$ path in $G_i$ (in particular,
$\pi_i(s_i)=\pi(s_i)$). Then $\pi_i$ satisfies $\ell(e)=\pi_i(v)-\pi_i(u)$ for
each $e=(u,v)\in E_i$.

We assert that for any $i,j\in I$, the potentials $\pi_i$ and $\pi_j$ coincide
on $G_i\cap G_j$.

Indeed, unless $s_i, s_j$ are antipodal (in which case $G_i\cap G_j$ consists
of two isolated terminals $s_i=t_j$ and $s_j=t_i$ and the assertion is
trivial), the region $\Omega_i\cap\Omega_j$ lies between the shortest paths
$L(s_jt_i)$ and $D(s_jt_i)$ (in view of~\refeq{casei}). This implies that the
subgraphs of $G_i$ and $G_j$ lying in $\Omega_i\cap \Omega_j$ are the same and
equal to $G_i\cap G_j$. Moreover, the latter is just the union of shortest
$s_j$--$t_i$ paths: $G_i\cap G_j=\cup(P\in\Pscr(s_jt_i)$). Since
$\pi_i(s_j)=\pi(s_j)=\pi_j(s_j)$ and $\pi_i(t_i)=\pi_j(t_i)$, the potentials
$\pi_i$ and $\pi_j$ must coincide on each shortest $s_j$--$t_i$ path, yielding
the result.
\end{proof}

We call the graph $N=N_H$ defined in this proposition the \emph{necklace} for
$H$. (Depending on the context in what follows, we may also think of the
necklace as the underlying undirected graph.) Two examples are illustrated in
the picture; here for simplicity all edges have the same length.

 \vspace{-0.2cm}
\begin{center}
\includegraphics[scale=0.9]{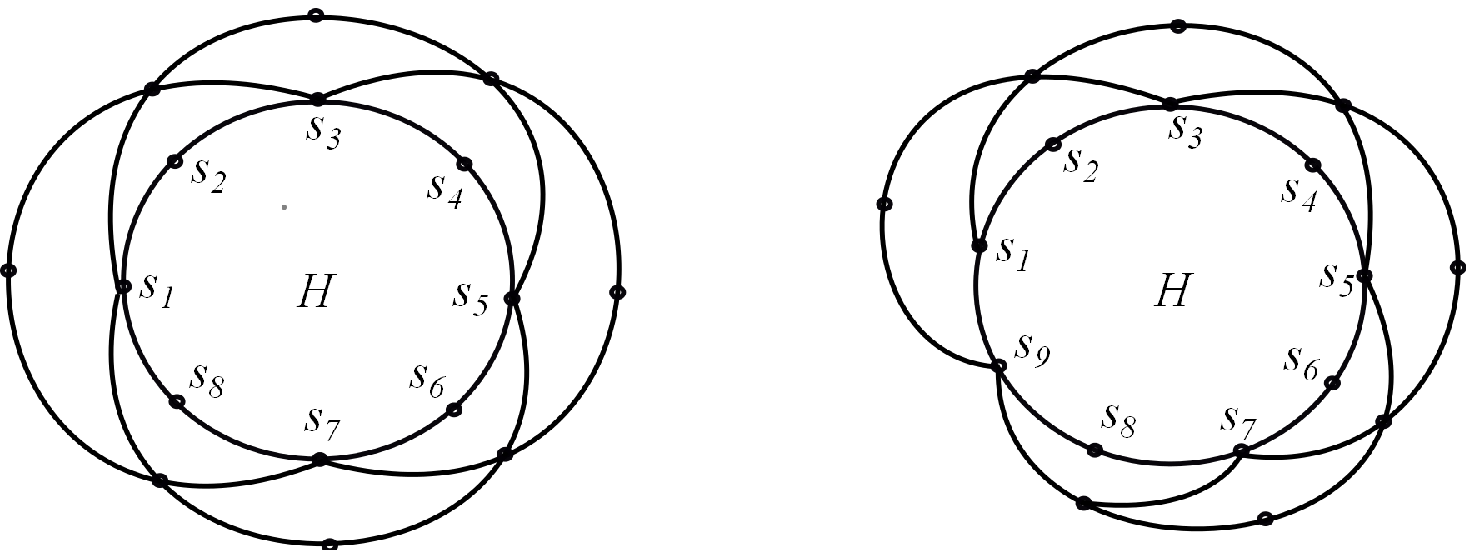}
\end{center}
\vspace{-0.2cm}

Denoting the set of all $H$-geodesics of type $0$ by $\Pscr^0_H$, we can
summarize the above observations and results as follows.
 \begin{corollary} \label{cor:necklace}
The initial problem can be reduced, in strongly polynomial time, to PMP with
$(G,\ell,\Hscr)$ satisfying (C1)--(C4) and the next property:
\begin{description}
\item{\rm(C5)} ~for each $H\in\Hscr$, the subgraph $\cup (P\in\Pscr^0_H)$ of $G$
can be directed so that any simple directed cycle separates $H$ from
$\Hscr-\{H\}$ and has $\ell$-length exactly $\sigma_H$.
 \end{description}
 \end{corollary}

We denote the set of such cycles in $N_H$ embracing $H$ by $\Cscr_H$. For
$C\in\Cscr_H$, let $\Omega(C)$ denote the closed region bounded by $C$ and
containing $H$. The correspondence $C\mapsto\Omega(C)$ leads to representing
$\Cscr_H$ as a distributive lattice with operations $\wedge,\vee$ defined by
the relations: for $C,C'\in\Cscr_H$, ~$\Omega(C\wedge
C')=\Omega(C)\cap\Omega(C')$ and $\Omega(C\vee C')=\Omega(C)\cup\Omega(C')$.
Then $\bd(H)$ is the minimal element of $\Cscr_H$, and we denote the maximal
element in it as $D_H$; so
  $$
  \ell(D_H)=\sigma_H \quad\mbox{and}\quad \Omega(C)\subseteq \Omega(D_H)
       \;\;\mbox{for all $C\in\Cscr_H$}.
  $$

Note that~\refeq{barrier} implies the following property:
 \begin{numitem1} \label{eq:distHH}
for distinct $H,H'\in\Hscr$, no $H'$-geodesic meets $\inter(\Omega(D_H))$.
  \end{numitem1}

We finish this section with one important special case.
\smallskip

\noindent\textbf{Definition.} ~The necklace $N_H$ is called \emph{trivial} if
$N_H=\bd(H)$.

\begin{prop} \label{pr:Htype0}
~If $\tau(H)=0$ then $N_H$ is trivial.
  \end{prop}
(Note that the converse need not hold.)
\smallskip

 \begin{proof}
Suppose that this is not so. Then $D_H\ne\bd(H)$ and $\Omega(D_H)$ contains a
face $F\ne H$ of $G$ (which lies in a face of $N_H$ but need not coincide with
the latter). We can choose two vertices $u,v$ in $\bd(F)$ not contained in a
directed path of $N_H$ (taking into account that $F$ has no dominating edge,
by~(C3)). Then at least one of these vertices, $v$ say, is not in $D_H$.
By~(C2), $u$ and $v$ belong to an $\Hscr$-geodesic $Q$. By~\refeq{distHH}, $Q$
cannot be an $H'$-geodesic for $H'\ne H$. So $Q$ is an $H$-geodesic. Moreover,
$\tau(Q)=0$. Then $Q$ turns into a directed path in $N_H$ containing $u$ and
$v$; a contradiction.
\end{proof}


\section{Elimination of inner edges} \label{sec:edge_free}

In this section we demonstrate one more sort of reducible cuts, aiming to
obtain the following result (as a weakened version of
Proposition~\ref{pr:without_faces}).

\begin{prop} \label{pr:inner_edges}
When $|\Hscr|=3$, one can find, in strongly polynomial time, a reducible
collection of cuts with integer weights so that the reduction by these cuts
results in a triple $(G',\ell',\Hscr')$ with $|\Hscr'|\le 3$ and $G'$ having no
inner edges (i.e., covered by the boundaries of holes).
  \end{prop}

Let $(G,\ell,\Hscr)$ be as in Corollary~\ref{cor:necklace}. In what follows,
until otherwise is explicitly said, we assume that
 \begin{numitem1} \label{eq:1edge}
$G$ has an inner edge or an inner face $F$ with $|V_F|\ge 4$.
  \end{numitem1}

\begin{lemma} \label{lm:nontrivNH}
Suppose that for some $H\in\Hscr$, the necklace $N_H$ is nontrivial. Then the
other two holes are of type 0 (and their necklaces are trivial, by
Proposition~\ref{pr:Htype0}).
  \end{lemma}
 \begin{proof}
Let $v$ be a vertex in $ bd(H)-D_H$ (existing since $\bd(H)\ne D_H$). One can
see that $\Omega(D_H)$ contains an inner face $F$ of $G$ and a vertex $u\ne v$
such that
\begin{numitem1} \label{eq:uvbdF}
both $u,v$ belong to $F$ but not connected by a directed path in $N_H$.
  \end{numitem1}

By~(C2), $u$ and $v$ belong to an $\Hscr$-geodesic $Q$. By~\refeq{distHH}, $Q$
is an $H$-geodesic, and~\refeq{uvbdF} implies that $\tau(Q)=1$, i.e., $Q$
separates the holes $H'$ and $H''$, where $\Hscr=\{H,H',H''\}$. One may assume
that the $u$--$v$ part of $Q$ lies in $\Omega(D_H)$.

Now suppose that $\tau(H')=1$ as well. Take an $H'$-geodesic $Q'$ separating
$H$ and $H''$. Then $Q$ and $Q'$ ``cross'' each other; let $x$ and $y$ be the
first and last vertices of $Q'$ occurring in $Q$, respectively. Let $R$ and
$R'$ be the $x$--$y$ parts of $Q$ and $Q'$, respectively; see the picture where
$R,R'$ are drawn in bold.

\vspace{-0cm}
\begin{center}
\includegraphics[scale=0.9]{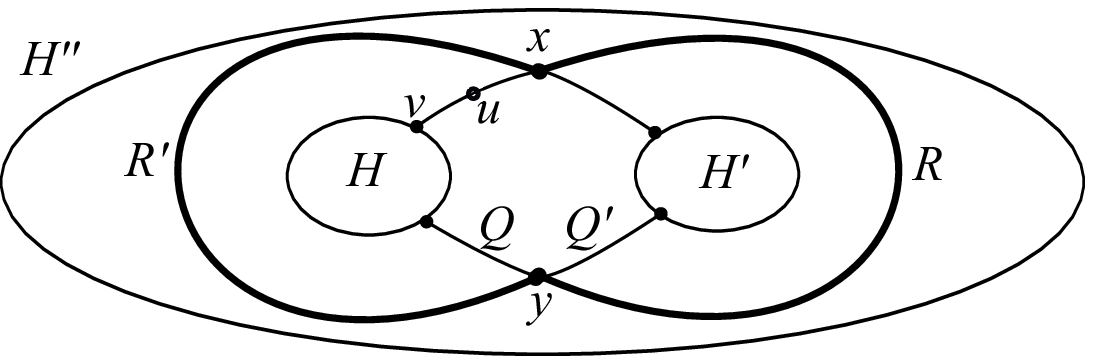}
\end{center}
\vspace{-0cm}

Exchange in $Q,Q'$ the pieces $R,R'$, forming $H$-path $\tilde Q$ and $H'$-path
$\tilde Q'$, respectively. In view of $\ell(\tilde Q)+\ell(\tilde Q')=
\ell(Q)+\ell(Q')$, both $\tilde Q,\tilde Q'$ are shortest. Also $\tilde Q$ does
not separate $H'$ and $H''$ and contains both vertices $v,u$. So $\tilde Q$ is
an $H$-geodesic of type 0 passing $v$ and $u$, which contradicts~\refeq{uvbdF}.
  \end{proof}

This lemma is generalized as follows.

  \begin{lemma} \label{lm:oneHtype1}
~Subject to~\refeq{1edge}, exactly one hole has type 1.
 \end{lemma}
\begin{proof}
~In view of lemma~\ref{lm:nontrivNH}, we may assume that the necklaces of all
holes are trivial. From ~\refeq{1edge} it follows that there are two vertices
$u,v$ contained in an inner face $F$ but not in the boundary of one hole.
By~(C2), for some $H\in\Hscr$, there is an $H$-geodesic $Q$ passing $u,v$.
Since $N_H$ is trivial, $\tau(Q)=\tau(H)=1$. Suppose that there is another hole
$H'$ of type 1. Choose an $H'$-geodesic $Q'$ of type 1. Then (like in the proof
of Lemma~\ref{lm:nontrivNH}) there are vertices $x,y\in V_Q\cap V_{Q'}$ such
that exchanging in $Q,Q'$ the pieces $R:=Q[x,y]$ and $R':=Q'[x,y]$, we obtain
an $H$-geodesic $\tilde Q$ and an $H'$-geodesic $\tilde Q'$, both of type 0.
Then $\tilde Q\subset \bd(H)$ and $\tilde Q'\subset\bd(H')$, and therefore
neither $\tilde Q$ nor $\tilde Q'$ contains both $u,v$, or, equivalently, $R$
contains exactly one of $u,v$. Assuming, w.l.o.g., that $R\cap R'=\{x,y\}$, one
can see that the cycle formed by $R'$ and the complement of $R$ to $\bd(H')$
separates $u$ from $v$. This is impossible since $u,v$ belong to one face.
 \end{proof}

 \begin{lemma} \label{lm:HHpconnect}
For any two holes $H,H'\in\Hscr$, ~$\bd(H)\cap\bd(H')$ is connected (possibly
empty).
 \end{lemma}
 \begin{proof}
Suppose this is not so. Then there are two paths $L\subset\bd(H)$ and
$L'\subset\bd(H')$ that have the same ends, $x,y$ say, and no common
intermediate vertices, and one of the two regions of the plane bounded by
$L\cup L$, ~$\Omega$ say, contains at least one inner face and no hole (in view
of $|\Hscr|=3$).

Each pair of vertices in a face within $\Omega$ belongs to an $\Hscr$-geodesic
lying in $\Omega$ and having both ends in one of $L,L'$. Considering such
geodesics, one can conclude that the subgraph of $G$ lying in $\Omega$ is
contained in $N_H\cup N_{H'}$. This easily implies $\ell(L)=\ell(L')$ and
$L,L'\subset N_H\cap N_{H'}$, contradicting~\refeq{barrier}.
  \end{proof}

Consider the auxiliary graph $\Gamma$ whose vertices are the inner faces of $G$
and whose edges are the pairs of inner faces sharing an edge. For a component
$\Gamma'$ of $\Gamma$, let $\Omega_{\Gamma'}$ be the union of faces that are
the vertices of $\Gamma'$. Lemma~\ref{lm:HHpconnect} implies that $\Gamma$ has
at most two components, and for each component $\Gamma'$, the region
$\Omega_{\Gamma'}$ is surrounded by three paths in the boundaries of holes. One
more important fact is as follows.
 \begin{lemma} \label{lm:Gammap}
Let an inner face $F$ share an edge $e$ with the hole $H$ of type 1. Suppose
that at least one is true: $F$ has an inner edge, or $|V_F|\ge 4$
(cf.~\refeq{1edge}). Then there is an endvertex $x$ of $e$ and a vertex $z\in
V_F$ such that both $x,z$ belong to an $H$-geodesic $P$ of type 1 and satisfy
$d(xz)+d(xy)>d(yz)$, where $y$ is the other endvertex of $e$.
  \end{lemma}
 \begin{proof}
Let $x_0,x_1,\ldots,x_k$ be the sequence of vertices in $\bd(F)$ and
$e=x_0x_k$. One may assume that $e$ is directed from $x_0$ to $x_k$ in $N_H$.
Consider two cases.
\smallskip

\emph{Case 1}: ~$F\not\subset \Omega(D_H)$. Then there exists $x_i$ which is
not in $\Omega(D_H)$. (Since $V_F\subset \Omega(D_H)$ together with
$F\not\subset \Omega(D_H)$ is possible only if $V_F\subset D_H$, and therefore
for some $j$, the sequence $x_jx_{j-1}\cdots x_0x_kx_{k-1}\cdots x_{j+1}$ forms
a directed path $Q$ in $D_H$. But then $\ell(Q)=\pi(x_{j+1})-\pi(x_j)$, whence
the edge $x_jx_{j+1}$ is dominating, contrary to~(C3).) Let $i$ be minimal
subject to $x_i\notin\Omega(D_H)$. If there is no hole $H'$ such that
$x_0,x_i\in V_{H'}$, then we assign $x:=x_0,y:=x_k,z:=x_i$.

Suppose that $x_0,x_i\in V_{H'}$ for some $H'\in\Hscr$. Then $H'\ne H$ and the
path $x_0x_1\cdots x_i$ lies in $\bd(H')$ (since $\tau(H')=0$, by
Lemma~\ref{lm:oneHtype1}). This implies that $i=1$ (taking into account that
$x_{i-1}$ is in $D_H$ and $x_0\in V_H\cap V_{H'}$). If there is no hole $H''$
such that $x_1,x_k\in V_{H''}$, then we assign $x:=x_k,y:=x_0,z:=x_1$.

Next suppose that $x_0,x_1\in V_{H'}$ and $x_1,x_k\in V_{H''}$, where
$\Hscr=\{H,H',H''\}$. Then the path $x_1x_2\cdots x_k$ lies in $\bd(H'')$. So
$\bd(F)$ has no inner edges, implying $|V_F|=k\ge 4$ (by the hypotheses of the
lemma). It follows that the pair $x_0,x_{k-1}$ is contained in none of the
holes, and we assign $x:=x_0,y:=x_k,z:=x_{k-1}$.

We assert that in all cases an $\Hscr$-geodesic $P$ passing $x,z$ as above
(existing by~(C2)) is as required. Indeed, $P$ is an $H$-geodesic (since any
$H'$-geodesic for $H'\ne H$ lies in $\bd(H')$, but $\{x,z\}\not\subset
V_{H'}$). Also $\tau(P)=0$ is impossible, in view of $z\notin\Omega(D_H)$.

It remains to show that the inequality $d(xz)+d(xy)\ge d(yz)$ is strict.
Supposing that it holds with equality, take an $\Hscr$-geodesic $Q$ passing
$y,z$. Then the path $Q'$ obtained from $Q$ by replacing its $y$--$z$ part by
the concatenation of the edge $e=yx$ and a shortest $x$--$z$ path is again an
$\Hscr$-geodesic. Moreover, $Q'$ is an $H$-geodesic of type 1 (by reasonings as
above and the choice of $x,z$). Then $Q'$ connects antipodal terminals in
$\bd(F)$, one of which is $y$; let $s$ be the other end of $Q'$. But the part
of $Q'$ with ends $x$ and $s$ is also an $H$-geodesic of type 1, and therefore
$x$ must be antipodal to $s$ as well; a contradiction.
\smallskip

\emph{Case 2}: ~$F\subset \Omega(D_H)$. Let $F'$ be the face of $N_H$
containing $e$ and different from $H$; then $F\subseteq F'$ (possibly $F=F'$).
Let $\bd(F')$ be formed by two directed $x'$--$y'$ paths $Q$ and $R$, and let
$e$ belong to $Q$. Since $e$ is non-dominating, at least one of $x'\ne x_0$ and
$y'\ne x_k$ is true. Let for definiteness $x'\ne x_0$. This implies $x_0\notin
D_H$. We assign $x:=x_0$, $y:=x_k$, and assign $z$ to be an arbitrary vertex
$x_i\in V_F$ contained in $R-\{x',y'\}$. Let $P$ be an $\Hscr$-geodesic passing
$x,z$. Then $P$ is an $H$-geodesic (in view of~\refeq{barrier} and
$x\in\Omega(D_H)-D_H$). Also $\tau(P)=1$ (since the pair $x,z$ does not belong
to a directed path in $N_H$). The required inequality $d(xz)+d(xy)>d(yz)$ is
proved as in Case~1.
\end{proof}

Consider $F,e,x,y,z,P$ as in this lemma and let $s$ be the endvertex of $P$
different from $x$. It is convenient for us to add to $G$ (for a while) extra
edges $e'=xz$ and $e''=yz$ with the lengths $\ell(e'):=d(xz)$ and
$\ell(e''):=d(yz)$, placing them in the face $F$ (unless such edges already
exist); this does not affect the problem. Accordingly replace in $P$ the
$x$--$z$ part by the edge $e'$. The updated path $P$ divides
$\Rset^2-\inter(H)$ into two closed regions $\Omega'$ and $\Omega''$ containing
$H'$ and $H''$, respectively. Let for definiteness $y$ lies in $\Omega''$; then
the face $T$ bounded by the edges $e,e',e''$ lies in $\Omega''$ as well; see
the picture.

\vspace{-0cm}
\begin{center}
\includegraphics[scale=0.9]{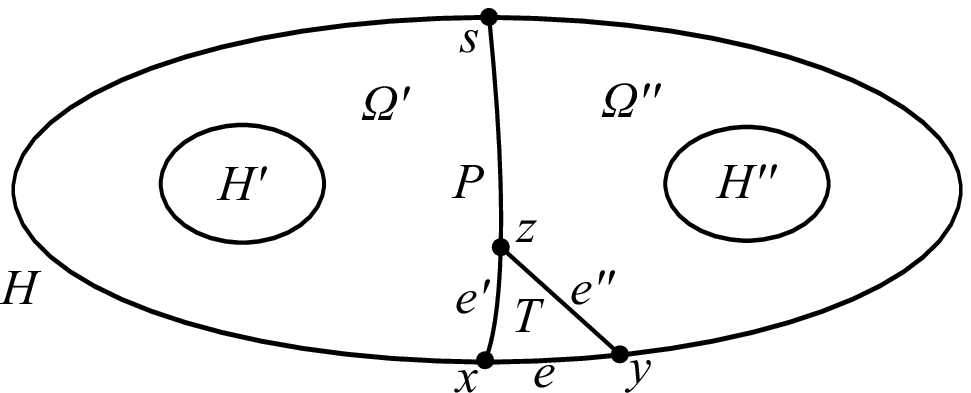}
\end{center}
\vspace{-0cm}

Let $P'$ be the $y$--$s$ path concatenating $e''$ and $P[z,s]$. We wish to
devise a collection of reducible cuts containing both edges $e,e'$ and
traversing $\Omega'$. To this aim, we extract from $G$ the subgraph
$G'=(V',E')$ lying in $\Omega'\cup T$ and, using the algorithm of~\cite{kar90},
solve the auxiliary two-hole PMP with $G'$, $\ell\rest{E'}$ and
$\Hscr':=\{H',\tilde H\}$, where $\tilde H$ is the face of $G'$ containing $H$.
It finds a packing $\Cscr$ of cuts $\delta X$ in $G'$ with integer weights
$\lambda(X)>0$ that realize the distances on $\Pi_{\Hscr'}$.

Let $\Cscr_1,\Cscr_2,\Cscr_3$ be the collections of those cuts $\delta
X\in\Cscr$ that contain $\{e,e'\}$, $\{e,e''\}$ and $\{e',e''\}$, respectively,
and define $a_i:=\sum(\lambda(X)\colon \delta X\in\Cscr_i)$, $i=1,2,3$. The
edges of $T$ must be saturated by $(\Cscr,\lambda)$, i.e., we have the
equalities
  $$
 d(xy)=a_1+a_2,\quad d(xz)=a_1+a_3,\quad d(yz)=a_2+a_3.
  $$
Hence $2a_1=d(xy)+d(xz)-d(yz)>0$, in view of Lemma~\ref{lm:Gammap}. For each
$\delta X\in\Cscr_1$, we may assume that $x\in X$. Since $\delta X$ can meet
any $\tilde H$-geodesic at most once, $e'$ is the unique common edge of $\delta
X$ and $P$. Then $X$ does not meet $P'$ and therefore $\delta X$ is a cut in
the whole $G$.
 \begin{lemma} \label{lm:C1lambda1}
Let $\ell':=\ell-\sum(\lambda(X)\chi^{\delta X}\colon \delta X\in\Cscr_1)$ and
$\lambda_1:=\lambda\rest{\Cscr_1}$. Then $(\Cscr_1,\lambda_1)$ is reducible for
$(G,\ell,\Hscr)$, i.e., any $p$--$q$ path $Q$ in $G$ with $pq\in\Pi_{\Hscr}$
satisfies~\refeq{QC1}.
  \end{lemma}

 \begin{proof}
We denote by $d'$ the distance in $(G',\ell'\rest{E'}$). For $Q$ as above, let
$R_1,\ldots, R_k$ be the components of $Q\cap P$, occurring in this order in
$Q$. We use induction on $k$. If $Q$ is entirely contained in one of
$\Omega',\Omega''$ (in particular, if $k=0$), then~\refeq{QC1} is easy (taking
into account that $(\Cscr_1,\lambda_1)$ is reducible for $(G',\ell',\Hscr')$,
whence $d'(e)=\ell'(e)=a_2$, $d'(e')=\ell'(e')=a_3$, $d'(e'')=d(e'')=a_2+a_3$).

In case $k\ge 2$, choose a vertex $u$ in $R_1$ and a vertex $v$ in $R_2$, and
let $\tilde Q:=Q[u,v]$ and $\tilde P:=P[u,v]$. The reducibility of
$(\Cscr_1,\lambda_1)$ for $(G',\ell\rest{E'},\Hscr')$ implies
$d'(uv)=\ell'(\tilde P)$. In its turn, $\ell(\tilde Q)\ge d'(u,v)$ when $\tilde
Q$ lies in $\Omega'$. And when $\tilde Q$ lies in $\Omega''$, we have
$\ell'(\tilde Q)=\ell(\tilde Q)\ge d(uv)$ if $x\ne u,v$, and $\ell'(\tilde
Q)=\ell(\tilde Q)-a_1\ge d'(uv)$ otherwise (since in the latter case $\tilde P$
contains the edge $e'$ and $\tilde Q$ must contain either $e$ or $e'$). Hence
$\ell'(\tilde Q)\ge \ell'(\tilde P)$ always hold, and~\refeq{QC1} follows by
induction (by replacing $\tilde Q$ by $\tilde P$ in $Q$).

It remains to consider the situation when $k=1$ and $Q$ meets both $\Omega'-P$
and $\Omega''-P$. Let $v\in V_Q\cap V_P$, ~$Q':=Q[p,v]$ and $Q'':=Q[v,q]$. One
may assume that $Q'\subset \Omega'$ (and $Q''\subset \Omega''$); then $p$ is in
$\Omega'-P$ and $q$ is in $\Omega''-P$. This implies $pq\in\Pi_H$.

Take in $(G,\ell)$ a shortest $p$--$v$ path $A$ and a shortest $v$--$q$ path
$B$. One may assume that $A\subset \Omega'$ and $B\subset \Omega''$. Since
$pv\in\Pi_{\tilde H}$, ~$A$ is shortest in $(G',\ell'\rest{E'})$, whence
$\ell'(Q')\ge\ell'(A)$. Also one can see that $\ell'(Q'')\ge \ell'(B)$
(considering both cases $v=x$ and $v\ne x$). So we can replace $Q$ by the
concatenation $L$ of $A$ and $B$. Then $\ell(L)=d(pv)+d(vq)\ge d(pq)$. Since
each cut in $\Cscr_1$ meets $A$ at most once,
  \begin{equation} \label{eq:ellA}
 \ell'(A)=d(pv)-a,
  \end{equation}
where $a:=\sum(\lambda(X)\,|\{pv\}\cap \rho_{\Hscr'} X|\,\colon \delta X\in
\Cscr_1)$.

If $v\ne x$ then, obviously, no cut in $\Cscr_1$ meets $B$, whence
$\ell(B)=\ell(B)=d(vq)$ and~\refeq{ellA} implies~\refeq{QC1} (with $L$ in place
of $Q$). And if $v=x$ then one may assume that both $A,B$ lie in $\bd(H)$. Then
$B$ (and therefore $L$) contains the edge $e$, whence
  \begin{equation} \label{eq:ellB}
 \ell'(B)=d(xq)-a_1.
  \end{equation}

Note that~\refeq{ellA} and~\refeq{ellB} immediately imply~\refeq{QC1} if $a=0$.
So assume that $a>0$. Each cut $\delta X\in\Cscr_1$ contributing to $a$
contains the edge $e'$ and one edge, $\hat e$ say,  of $A$, whereas each cut
non-contributing to $a$ contains $e'$ and one edge, $\tilde e$ say, in the
corresponding $p$--$s$ path $R$ in $\bd(\tilde H)$ (here $\hat e,\tilde e$ are
''opposite'' (in a sense) to $e'$ in the cycle $P\cup A\cup R$). The relation
$\ell(P)=\ell(A)+\ell(R)=\sigma_H/2$ implies that the latter cuts use the
entire $\ell$-length of $R$. Then $\ell(R)=a_1-a$, and we obtain
$\ell(A)=\sigma_H/2-a_1+a$. This together with $\ell(B)\ge \ell(e)\ge a_1$
gives
  $$
  \ell(L)=\ell(A)+\ell(B)\ge\sigma_H/2+a.
  $$
Then the $p$--$q$ path $L'\in \Lscr_H(pq)$ different from $L$ satisfies
$\ell(L')=\sigma_H-\ell(L)\le\sigma_H/2-a$. This implies $d(pq)=\ell(L')\le
\sigma_H/2-a$ and $d(px)+d(xq)=\ell(L)\ge d(pq)+2a$. Now adding~\refeq{ellA}
and~\refeq{ellB}, we have
  \begin{equation} \label{eq:ellpL}
 \ell'(L)=\ell'(A)+\ell'(B)=(d(px)-a)+(d(xq)-a_1)\ge d(pq)+a-a_1.
  \end{equation}

But $a_1-a$ is equal to the sum of values $\lambda(X)$ over the cuts $\delta
X\in\Cscr_1$ separating $p$ and $q$. So~\refeq{ellpL} implies the required
relation~\refeq{QC1}.
 \end{proof}

\noindent \textbf{Reduction III: Implementation and convergency.} We refer to a
natural procedure behind Lemmas~\ref{lm:Gammap} and~\ref{lm:C1lambda1} as
\emph{Reduction III}. It scans all $F,e$ as in Lemma~\ref{lm:Gammap} and, at a
current iteration, finds corresponding $x,y,z,P$ for $F,e$, compute
$(\Cscr_1,\lambda_1)$ and reduce $\ell$ to $\ell'$ as in
Lemma~\ref{lm:C1lambda1}, after which the extra edges $e',e''$ (if exist) are
deleted. Note that
\begin{numitem1} \label{eq:ddp}
~$d(yz)<d(xy)+d(xz)$ turns into $d'(yz)=d'(xy)+d'(xz)$, where
$d':=d_{G,\ell'}$.
  \end{numitem1}
Then we update $\ell:=\ell'$, and so on until $F,e$ as in Lemma~\ref{lm:Gammap}
no longer exist.

The process finishes in $O(|V|^3)$ iterations. To see this, consider a current
iteration and use notation as above. Assume that $\ell'(e)=a_2>0$ (otherwise
$e$ is contracted). Take an $\Hscr$-geodesic $C$ in $(G,\ell)$ passing $y,z$
and let $C'$ be formed from $C$ by replacing $C[y,z]$ by the concatenation of
$e$ and an $\ell$-shortest $x$--$z$ path. By~\refeq{ddp}, $C'$ is
$\ell'$-shortest but not $\ell$-shortest.

Next we argue as follows (cf. the proof of Lemma~\ref{lm:Gammap}). If $C'$ is
an $H'$-geodesic for $\ell'$, then the trivial necklace for $(H',\ell)$
captures $C'$ and becomes nontrivial (since $x\notin V_{H'}$). When $C'$ is an
$H''$-geodesic (for $\ell'$), the behavior is similar. Now let $C'$ be an
$H$-geodesic for $\ell'$. Then $\tau(C')=0$ (for otherwise $C'$ connects
antipodal terminals $y$ and $s$ in $\bd(H)$, and similarly for its part
$C'[x,s]$; this is impossible since $\ell'(e)=a_2>0$). In this case, $N_H$
captures at least one new edge (in view of~\refeq{ddp}).

Thus, in all cases at least one of $N_H,N_{H'},N_{H''}$ grows. This is possible
only if for some $v\in V$ and $st\in\Pi_{\Hscr}$, the excess $\eps(v|st)$
changes from a positive value to zero.

This yields the desired complexity and completes the proof of
Proposition~\ref{pr:inner_edges}. Moreover, we have shown that upon termination
of Reduction~III,
  \begin{numitem1} \label{eq:two-inner-faces}
$G$ has at most two inner faces, and each inner face $F$ has exactly three
edges and shares one edge with each hole.
  \end{numitem1}


\section{Final reductions and the proof of Theorem~\ref{tm:H3alg}} \label{sec:final}

To finish the proof of the main theorem, it remains to consider
$G,\ell,\Hscr=\{H_1,H_2,H_3\}$ satisfying~\refeq{two-inner-faces}. An example
with two inner faces $F,F'$ is illustrated in the left fragment of the picture;
here $E_F=\{e_i=v_{i-1}v_{i+1}\, \colon i=1,2,3\}$ and $F$ shares the edge
$e_i$ with $H_i$ (taking indices modulo 3).

\vspace{-0cm}
\begin{center}
\includegraphics[scale=1.0]{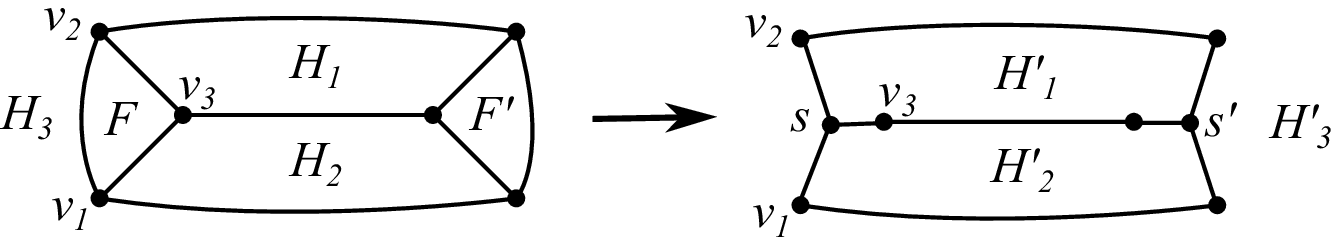}
\end{center}
\vspace{-0cm}

Let us replace the ``triangle'' $F$  by a star. More precisely, insert in
$\inter(F)$ one extra vertex $s$, delete the edges of $\bd(F)$ and connect $s$
with each $v_i$ by an edge with the length $a_i$ such that
   $$
 a_{i-1}+a_{i+1}=d(v_{i-1}v_{i+1})\;\; (=\ell(e_i)), \quad i=1,2,3;
   $$
these lengths are unique, nonnegative and integral. Act similarly for the other
inner face $F'$ (if exists). This transformation results in a graph with three
holes and without inner faces; see the right fragment of the above picture.
Also the lengths of edges in this graph are cyclically even and preserve the
distances between the old vertices. Then a solution for the new triple (keeping
notation $(G,\ell,\Hscr)$ for it) determines a solution for the old one,
yielding Proposition~\ref{pr:without_faces}.

Thus, we come to the case when the current graph $G$ is formed by three openly
disjoint paths $P_1,P_2,P_3$ connecting two vertices $s$ and $s'$, and each
hole $H_i$ is bounded by $P_{i-1}\cup P_{i+1}$. Assume that $\ell(P_1)\le
\ell(P_2)\le\ell(P_3)$. Note that if $\ell(P_1)<\ell(P_3)$ then $P_1$ is an
$H_1$-geodesic of type 1 and the pair $(P_1,P_3)$ is excessive w.r.t. $H_1$
(causing Reduction~I from Sect.~\ref{sec:red_cuts}). So we may assume that
$P_1,P_2,P_3$ have the same length $b$.

It is convenient to slightly modify $(G,\ell)$ to make it ``path-invariant''
and ``mirror-reflective''. More precisely, subdividing some edges of $G$ (which
adds at most $5|V|$ new vertices), we can update $G$ so that for any vertex $x$
of a path $P_i$, each other path $P_j$ has vertex $y$ with $d(sy)=d(sx)$ and
vertex $z$ with $d(sz)=b-d(sx)$. Let $s=x^i_0,x^i_1,\ldots, x^i_k=s'$ be the
sequence of vertices in $P_i$, $i=1,2,3$, and define
  $$
  \lambda_r:=\ell(x^i_{r-1}x^i_r)\;\;(=d(sx^i_r)-d(sx^i_{r-1})), \quad
r=1,\ldots, k,
  $$
which does not depend on $i$ and satisfies $\lambda_r=\lambda_{k-r+1}$.

Now we represent $d$ as the sum of weighted (2,3)-metrics and, possibly, one
cut-metric as follows:

 \begin{numitem1} \label{eq:23metrics}
for $r=1,2,\ldots,\lfloor k/2\rfloor$, define $m_r$ to be the (2,3)-metric on
$V$ determined by the partition $(S^1_r,S^2_r;\, T^1_r,T^2_r,T^3_r)$, where

$S^1_r:=\{x^i_p\, \colon i=1,2,3,\, p=0,1,\ldots,r-1\}$;

$S^2_r:=\{x^i_p\, \colon i=1,2,3,\, p=k-r+1,\ldots,k\}$;

$T^j_r:=\{x^j_p\,\colon p=r,\ldots, k-r\}$,\; $j=1,2,3$.
  \end{numitem1}

\noindent(Recall that the (2,3)-metric determined by a partition $(S^1,S^2;\,
T^1,T^2,T^3)$ of $V$ is the one induced by the map $\gamma:V\to V(K_{2,3})$
with $S^i=\gamma^{-1}(s_i)$ and $T^j=\gamma^{-1}(t_j)$, where $\{s_1,s_2\}$ and
$\{t_1,t_2,t_3\}$ are the color classes of $K_{2,3}$.) Also when $k$ is odd, we
assign as $m_{\lceil k/2\rceil}$ the cut metric on $V$ associated with the cut
$\delta X$ for $X:=\{x^i_p\,\colon i=1,2,3,\, p=0,1,\ldots,\lfloor
k/2\rfloor\}$. Each metric $m_r$ is endowed with the weight $\lambda_r$.

A routine verification shows that
  $$
 d=\lambda_1m_1+\lambda_2m_2+\cdots +\lambda_{\lceil k/2\rceil}m_{\lceil k/2\rceil},
  $$
yielding a solution to PMP for $(G,\ell,\Hscr)$. This completes the proof of
Theorem~\ref{tm:H3alg}.
 \medskip

\noindent\textbf{Acknowledgement.} I thank Maxim Babenko for useful
discussions.

\end{document}